\documentclass[final,1p,times]{elsarticle}

\usepackage{amsthm,amsmath,amssymb,amsfonts,graphicx,graphics,latexsym,exscale,cmmib57,dsfont,bbm,amscd,ulem}
\usepackage{mathrsfs}
\usepackage{color,soul}
\usepackage[colorlinks=true]{hyperref}

\newcommand{\N}{\boldsymbol{N}}
\newcommand{\F}{\boldsymbol{F}}
\newcommand{\X}{\boldsymbol{X}}
\newcommand{\T}{\boldsymbol{T}}

\newcommand{\x}{\boldsymbol{x}}

\newcommand{\bvarphi}{\boldsymbol{\varphi}}
\newcommand{\bX}{\boldsymbol{\pmb{\mathscr{X}}}}

\newcommand{\bea}{\begin{eqnarray}}
\newcommand{\eea}{\end{eqnarray}}
\newcommand{\bean}{\begin{eqnarray*}}
\newcommand{\eean}{\end{eqnarray*}}
\newtheorem{Thm}{Theorem}[section]
\newtheorem{cor}[Thm]{Corollary}
\newtheorem{prop}[Thm]{Proposition}
\newtheorem{Lem}[Thm]{Lemma}
\newtheorem*{AA-I}{Arzel\'{a}-Ascoli Theorem}
\newtheorem*{thm0.4}{Theorem~\ref{thm0.4}}

\theoremstyle{definition}
\newtheorem{defn}[Thm]{Definition}
\newtheorem*{remark-of-thm0.4}{Remarks of Theorem~\ref{thm0.4}}
\newtheorem*{remark-of-thm0.5}{Remarks of Theorem~\ref{thm0.5}}
\newtheorem*{remark-of-thm0.6}{Remarks of Theorem~\ref{thm0.6}}
\newtheorem*{remark-of-thm0.8}{Remarks of Theorem~\ref{thm0.8}}

\newtheorem{que}[Thm]{Question}

\newtheorem{remark}[Thm]{Remark}

\numberwithin{equation}{section}
\journal{Nonlinearity (arXiv: 1604.00611v1 [math.DS], April 5)}

\begin{document}
\begin{frontmatter}

\title{\sc Pointwise convergence of ergodic averages of bounded measurable functions for amenable groups}

\author{Xiongping Dai}
\ead{xpdai@nju.edu.cn}
\address{Department of Mathematics, Nanjing University, Nanjing 210093, People's Republic of China}

\begin{abstract}
Given any amenable group $G$ (with a left Haar measure $|\cdot|$ or $dg$), we can select out a \textit{F{\o}lner subnet} $\{F_\theta;\theta\in\Theta\}$ from any left F{\o}lner net in $G$, which is \textit{$L^\infty$-admissible} for $G$, namely, for any Borel $G$-space $(X,\mathscr{X})$ and any $\varphi\in L^\infty(X,\mathscr{X})$, there exists some function $\varphi^*$ on $X$ such that
\begin{gather*}
\textit{Moore-Smith-}\lim_{\theta\in\Theta}\frac{1}{|F_\theta|}\int_{F_\theta}\varphi(gx)dg=\varphi^*(x)\ \forall x\in X\quad {\textrm{and}}\quad
\varphi^*=(g\varphi)^*\ \forall g\in G.
\end{gather*}
Moreover, if $G$ is $\sigma$-compact such as a locally compact second countable Hausdorff amenable group, then $\varphi^*\in L^\infty(X,\mathscr{X})$, $\varphi^*(gx)=\varphi^*(x)$ \textit{a.e.}, and $\varphi^*$ is \textit{a.e.} independent of the choice of the $L^\infty$-admissible F{\o}lner net $\{F_\theta; \theta\in\Theta\}$ in $G$.

Consequently, we may easily obtain: (1) a Universal Ergodic Disintegration Theorem; (2) the Khintchine Recurrence Theorem for $\sigma$-compact amenable group; (3) the existence of $\sigma$-finite invariant Radon measures for any Borel action of an amenable group on a locally compact, $\sigma$-compact, metric space $X$ by continuous maps of $X$; and (4) an $L^\infty$-pointwise multiple ergodic theorem.
\end{abstract}

\end{frontmatter}

\section{Introduction}\label{sec0}
F{\o}lner sequences permit ergodic averages to be formed, and then the $L^p$-mean and $L^1$-pointwise ergodic theorems hold under ``suitable conditions'' for measure-preserving actions of $\sigma$-compact amenable groups; see \cite{Bew, Tem, Gre, Eme, EG, Fur, Shu, OW, Lin, EW, Nev, CF} and so on. Although the $L^1$-pointwise ergodic theorem has been well known over tempered F{\o}lner sequences since E.~Lindenstrauss 2001~\cite{Lin}, yet there has been no one over F{\o}lner net for general amenable group that does not have any F{\o}lner sequence at all. On the other hand, is there any other kind of admissible F{\o}lner sequences for pointwise convergence but the tempered one?

In this paper, we prove the $L^p$-mean ergodic theorem, $1\le p<\infty$, over any F{\o}lner net (cf.~Theorem~\ref{thm0.4}); and moreover, for any amenable group $G$, we can select out some F{\o}lner subnet from any F{\o}lner net in $G$, which is \textit{admissible} for the $L^\infty$-pointwise ergodic theorem for any Borel $G$-space $X$ (cf.~Theorem~\ref{thm0.5}). An important setting that is beyond the ergodic theorems available in the literature is the topological action on a compact metric space $X$:
\begin{itemize}
\item $G\curvearrowright X$, where $G$ is a \textit{discrete uncountable abelian} group.
\end{itemize}
But our results are valid for this case.

Consequently, as applications, we will present a new ergodic decomposition theorem for $\sigma$-compact amenable groups by only a concise one-page proof (cf.~Theorem~\ref{thm0.8}), the Khintchine recurrence theorem for any continuous action of $\sigma$-compact amenable groups (cf.~Theorem~\ref{thm0.9}), a existence of $\sigma$-finite invariant measure (cf.~Proposition~\ref{pro0.10}), and a simple $L^\infty$-pointwise convergence theorem of multiple ergodic averages for any topological module action on a measurable space (cf.~Theorem~\ref{thm0.11}).

\subsection{Amenability}\label{sec0.1}
By an \textbf{\textit{amenable}} group, in one of its equivalent characterizations, it always refers to a locally compact Hausdorff (lcH) group $G$, with a \textit{left} Haar measure $m_G$ (sometimes write $|\cdot|$ and $dg$ if no confusion) on the Borel $\sigma$-field $\mathscr{B}_G$ of $G$, for which it holds the left \textit{F{\o}lner condition}~\cite{Pat}: for any compacta $K$ of $G$ and $\varepsilon>0$, there exists a compacta $F$ of $G$ such that
$F$ is \textbf{\textit{$(K,\varepsilon)$-invariant}}; i.e., $|KF\vartriangle F|<\varepsilon|F|$,
where $A\vartriangle B=(A\setminus B)\cup(B\setminus A)$ is the symmetric difference.

Recall from \cite[Definition~4.15]{Pat}) that a net $\{K_\theta,\theta\in\Theta\}$ (where $(\Theta,\geqq)$ is a directed set) of compacta of positive Haar measure of $G$ is called a \textit{summing net} in $G$ if the following conditions are satisfied:
\begin{enumerate}
\item[(1)] $K_\theta\subseteq K_{\vartheta}$ if $\theta\leqq\vartheta$;
\item[(2)] $G=\bigcup_{\theta\in\Theta }K_\theta^{\circ}$ where $K^\circ$ is the interior of the set $K$;
\item[(3)] $\lim_{\theta\in\Theta}{|gK_\theta\vartriangle K_\theta|}\cdot|K_\theta|^{-1}=0$ uniformly for $g$ on compacta of $G$.
\end{enumerate}
Here $\lim_{\theta\in\Theta}$ is in the sense of Moore-Smith limit; cf.~\cite{Kel}. Clearly, if $G$ has a summing sequence it is $\sigma$-compact.

From now on let $G$ be an lcH group, then there holds the following very important lemma for our later arguments.

\begin{Lem}[{\cite[Theorem~4.16]{Pat}}]\label{lem0.1}
$G$ is amenable if and only if there
exists a summing net in $G$. $G$ is a $\sigma$-compact amenable group if and only
if there exists a summing sequence in $G$.
\end{Lem}

We will need the following two basic concepts.

\begin{defn}[\cite{Pat}]\label{def0.2}
Let $\{F_\theta; \theta\in\Theta\}$ be a net of compacta of positive Haar measure of $G$.
\begin{enumerate}
\item[(1)] $\{F_\theta; \theta\in\Theta\}$ is called a \textbf{\textit{F{\o}lner net}} in $G$ if and only if
\begin{itemize}
\item $\lim\limits_{\theta\in\Theta}\frac{|gF_\theta\vartriangle F_\theta|}{|F_\theta|}=0\ \forall g\in G$.
\end{itemize}
\item[(2)] $\{F_\theta; \theta\in\Theta\}$ is called a \textbf{$\mathscr{K}$-\textit{F{\o}lner net}} (or uniform F{\o}lner net) in $G$ if and only if
\begin{itemize}
\item $\lim\limits_{\theta\in\Theta}\frac{|KF_\theta\vartriangle F_\theta|}{|F_\theta|}=0$ for every compacta $K$ of $G$.
\end{itemize}
\end{enumerate}
One can analogously define \textbf{\textit{F{\o}lner sequence}} and \textbf{\textit{$\mathscr{K}$-F{\o}lner sequence}} in a $\sigma$-compact amenable group $G$.
\end{defn}

A $\mathscr{K}$-F{\o}lner sequence is simply called a F{\o}lner sequence in many literature like \cite{Lin, EW}. Clearly, every $\mathscr{K}$-F{\o}lner net/sequence in $G$ must be a F{\o}lner net/sequence; but the converse is never true if $G$ is not a discrete amenable group. In fact, one can easily construct a F{\o}lner sequence in $\mathbb{R}$, with the usual Euclidean topology, that is not a $\mathscr{K}$-F{\o}lner sequence (cf.~\cite[Problem~4.7]{Pat}).

We can then obtain the following fact from Lemma~\ref{lem0.1}.

\begin{Lem}\label{lem0.3}
Every amenable group $G$ has a ($\mathscr{K}$-)F{\o}lner net in it. Moreover, if $G$ is $\sigma$-compact, then it has a ($\mathscr{K}$-)F{\o}lner sequence.
\end{Lem}

Note that every subnet of a F{\o}lner net in $G$ is also a F{\o}lner net in $G$. A non-$\sigma$-compact amenable group does not have any F{\o}lner sequence in it. For example, $\mathbb{R}^d$, as a discrete abelian group, it is amenable but has no F{\o}lner sequence at all.

\subsection{Ergodic theorems for amenable groups}\label{sec0.2}
Let $(X,\mathscr{X})$ be a measurable space. By $L^\infty(X,\mathscr{X})$, it denotes the Banach space composed of all the bounded $\mathscr{X}$-measurable functions $\varphi\colon X\rightarrow\mathbb{R}$ with the uniform convergence topology induced by the sup-norm $\|\varphi\|_\infty=\sup_{x\in X}|\varphi(x)|$.

Let $G\curvearrowright_TX$ be a \textbf{\textit{Borel}} action of an amenable group $G$ on $X$, by $\mathscr{X}$-measurable maps of $X$ into itself, where $T\colon (g,x)\mapsto T_gx$ or $gx$ is the jointly measurable $G$-action map. In this case, $X$ is called a \textbf{\textit{Borel $G$-space}}. See $\S\ref{sec3.2}$

By $\mathcal{M}(G\curvearrowright_TX)$ we mean the set of all the $G$-invariant probability measures on $(X,\mathscr{X})$. It is well known that if $X$ is compact metrizable with $\mathscr{X}=\mathscr{B}_X$ and $T$ is continuous, then $\mathcal{M}(G\curvearrowright_TX)$ is non-void compact convex (cf.~\cite{Fom, Zim}).
In addition, independently of the Haar measure of $G$ we can define, for $1\le p<\infty$, two linear closed subspaces of $L^p(X,\mathscr{X},\mu)$:
\begin{subequations}
\begin{align}
\F_\mu^p&=\{\phi\in L^p(X,\mathscr{X},\mu)\,|\,T_g\phi=\phi\ \forall g\in G\}\\
\intertext{and}
\N_\mu^p&=\overline{\{\phi-T_g\phi\,|\,\phi\in L^p(X,\mathscr{X},\mu), g\in G\}},
\end{align}\end{subequations}
where, for any $g\in G$, the \textit{Koopman operator}
$$T_g\colon L^p(X,\mathscr{X},\mu)\rightarrow L^p(X,\mathscr{X},\mu)$$
associated to $G\curvearrowright_TX$ is defined by $T_g\phi\colon X\rightarrow X;\ x\mapsto\phi(T_gx)$ for $x\in X$.

We can then obtain the following $L^p$-mean ergodic theorem, which generalizes the $\sigma$-compact amenable group case.

\begin{Thm}\label{thm0.4}
Let $G\curvearrowright_TX$ be a Borel action of an amenable group $G$ on $X$ and let $\mu$ belong to $\mathcal{M}(G\curvearrowright_TX)$; then, for any $1\le p<\infty$,
\begin{subequations}
\begin{gather}
L^p(X,\mathscr{X},\mu)=\F_\mu^p\oplus\N_\mu^p.
\intertext{Moreover, for any F{\o}lner net $\{F_\theta; \theta\in\Theta\}$ in $G$ and any $\phi\in L^p(X,\mathscr{X},\mu)$,}
L^p(\mu)\textrm{-}\lim_{\theta\in\Theta}\frac{1}{|F_\theta|}\int_{F_\theta}T_g\phi dg=P(\phi).
\intertext{Here}
P\colon\F_\mu^p\oplus\N_\mu^p\rightarrow\F_\mu^p
\end{gather}
\end{subequations}
is the projection with $\|P\|\le1$. The statement also holds for $\sigma$-finite $\mu$ with $\mu(X)=\infty$ when $1<p<\infty$.
\end{Thm}

\begin{remark-of-thm0.4}
\begin{enumerate}
\item[(1)]For any Bore action of \textit{$\sigma$-compact} amenable group $G$, the $L^p$-mean ergodic theorem, particularly $L^2$-case, over any summing sequence or $\mathscr{K}$-F{\o}lner sequence $\{F_n\}$ in $G$, has been showed since Bewley 1971~\cite{Bew} and Greenleaf 1973~\cite{Gre}. Also see \cite{EG, Fur, KT, Kre, Shu, Lin, EW} etc. However, Theorem~\ref{thm0.4} has no other restriction on the F{\o}lner sequence/net $\{F_\theta; \theta\in\Theta\}$ in $G$.
\item[(2)]One step in the recipe for proving $L^1$-pointwise ergodic theorem (cf.~\cite[$\S2.4$]{Nev}) is based on the following ``fact'':
\begin{itemize}
\item \textit{$\frac{1}{|F_\theta|}\int_{F_\theta}T_g\phi(x)dg\to\phi(x)$ a.e. $\forall \phi\in\F_\mu^p$}.
\end{itemize}
This is because if $G$ is $\sigma$-compact, then for $\phi\in\F_\mu^p$, one can find by Fubini's theorem some $\mu$-conull $X_0\in\mathscr{X}$ such that $T_g\phi(x)=\phi(x)$ $m_G$-a.e. $g\in G$ for any $x\in X_0$ and thus
$$\frac{1}{|F_\theta|}\int_{F_\theta}T_g\phi(x)dg=\phi(x)\quad \forall x\in X_0.$$
However, if $G$ is not $\sigma$-compact, then $m_G$ is not $\sigma$-finite and $\{F_\theta; \theta\in\Theta\}$ is not countable. Thus the above ``fact'' is never a true fact.
\end{enumerate}
\end{remark-of-thm0.4}

Since the action map $T(g,x)$ is $\mathscr{B}_G$-measurable with respect to $g$ for any fixed $x$ in $X$, then for any compacta $K$ of $G$ and any observation $\varphi\in L^\infty(X,\mathscr{X})$, by Fubini's theorem we can well define the \textbf{\textit{ergodic average of $\varphi$ over $K$}} as follows:
\begin{gather*}
A(K,\varphi)(x)=\frac{1}{|K|}\int_{K}\varphi(T_gx)dg\quad \forall x\in X.
\end{gather*}
Unlike the mean ergodic theorem, we can only obtain the $L^\infty$-pointwise convergence over F{\o}lner subnet in $G$ here.

In this paper, we will only pay our attention to the following left-Haar-measure version which is more convenient for us to make use later on.

\begin{Thm}\label{thm0.5}
Let $\{F_{\theta^\prime}^\prime;\theta^\prime\in\Theta^\prime\}$ be any F{\o}lner net in any amenable group $G$. Then we can select out a F{\o}lner subnet $\{F_\theta; \theta\in\Theta\}$ of $\{F_{\theta^\prime}^\prime;\theta^\prime\in\Theta^\prime\}$, which is \textbf{$L^\infty$-admissible} for $G$; namely, if
$G\curvearrowright_TX$ is a Borel action of $G$ on any $X$, then there is a continuous linear function
\begin{subequations}
\begin{gather}
A(\centerdot)\colon (L^\infty(X,\mathscr{X}),\|\cdot\|_\infty)\rightarrow(\mathbb{R}^X,\mathfrak{T}_{\mathrm{p}\textrm{-}\mathrm{c}});\quad \varphi\mapsto\varphi^*
\intertext{such that: $\forall \varphi\in L^\infty(X,\mathscr{X})$,}
\varphi^*=(T_g\varphi)^*\quad  \forall g\in G
\intertext{and}
\lim_{\theta\in\Theta}A(F_{\theta},\varphi)(x)=\varphi^*(x)\quad \forall x\in X.\label{eq0.3c}
\end{gather}
\end{subequations}
If $G$ is abelian, then $T_g\varphi^*=\varphi^*$ for any $g\in G$.
\end{Thm}

\noindent
Here $\mathbb{R}^X$ is equipped with the product topology, i.e., the topology $\mathfrak{T}_{\textrm{p-c}}$ of the usual pointwise convergence.

\begin{remark-of-thm0.5}
\begin{enumerate}
\item[(1)]Here we cannot obtain that $\varphi^*(x)$ is $\mathscr{X}$-measurable, since $\varphi^*$ is only a Moore-Smith limit of a net of measurable functions $A(F_{\theta},\varphi)$; if $\lim_{\theta\in\Theta}$ is sequential, then the limit $\varphi^*$ is automatically $\mathscr{X}$-measurable for any $\varphi\in L^\infty(X,\mathscr{X})$ by the classical measure theory. Moreover, we cannot assert that $\varphi^*(x)$ is invariant too. However, this theorem is already useful for some questions; see, e.g., Theorem~\ref{thm0.6}, Theorem~\ref{thm0.8} and Proposition~\ref{pro0.10} below for ergodic theory.

\item[(2)]It should be noted here that it is illegal to arbitrarily take $\varphi\in L^\infty(X,\mathscr{X},\mu)$ in place of $\varphi\in L^\infty(X,\mathscr{X})$ in general, for $m_G$ does not need to be $\sigma$-finite and so we cannot employ Fubini's theorem on $G\times X$ here.

\item[(3)]The $L^p$-mean ergodic theorem will play no role for proving Theorem~\ref{thm0.5} because of (2) in Remarks of Theorem~\ref{thm0.4}.
\end{enumerate}
\end{remark-of-thm0.5}

The following is an easy consequence of Theorems~\ref{thm0.4} and \ref{thm0.5}, in which we will not impose any additional restriction, like the Tempelman condition and the Shulman condition (cf.~Remark (2) of Theorem~\ref{thm0.6} below), on the F{\o}lner sequences we will consider here.

\begin{Thm}\label{thm0.6}
Let $G$ be an amenable group; then we can select out a F{\o}lner subnet $\{F_\theta; \theta\in\Theta\}$ from any F{\o}lner net $\{F_{\theta^\prime}; \theta^\prime\in\Theta^\prime\}$ in $G$, which is \textbf{$L^\infty$-admissible} for $G$; namely, if $G\curvearrowright_TX$ is a Borel action of $G$ on any $X$,
then one can find a linear operator $\varphi\mapsto\varphi^*$ from $L^\infty(X,\mathscr{X})$ to $\mathbb{R}^X$ (to $L^\infty(X,\mathscr{X})$ if $G$ is $\sigma$-compact by taking $L^\infty$-admissible F{\o}lner sequence for $G$) such that for all $\varphi\in L^\infty(X,\mathscr{X})$,
\begin{itemize}
\item $\varphi^*=(T_g\varphi)^*\ \forall g\in G$,
\item $\lim_{n\to\infty}A(F_{n},\varphi)(x)=\varphi^*(x)\ \forall x\in X$;
\item moreover, if $\mathcal{M}(G\curvearrowright_TX)\not=\varnothing$, then for each $\mu\in\mathcal{M}(G\curvearrowright_TX)$
\begin{itemize}
\item $T_g\varphi^*(x)=\varphi^*(x)$, $\mu$-a.e., $\forall g\in G$,
\item $\varphi^*=P(\varphi)$ $\mu$-a.e., and
\item if $\mu$ is ergodic, then $\varphi^*(x)\equiv\int_X\varphi d\mu$ for $\mu$-a.e. $x$ in $X$.
\end{itemize}
\end{itemize}
Here $P(\centerdot)$ is the projection for $p=1$ as in Theorem~\ref{thm0.4}.
\end{Thm}

\begin{remark-of-thm0.6}
\begin{enumerate}
\item[(1)] The pointwise convergence at \textit{everywhere} in Theorem~\ref{thm0.6} cannot be extended to $L^1(X,\mathscr{X},\mu)$ in stead of $L^\infty(X,\mathscr{X})$. Let us consider the action of $\mathbb{Z}$ on the circle $\mathbb{T}=\mathbb{R}/\mathbb{Z}$ defined by
    $$T_n\colon x\mapsto x+n\alpha\quad (\textrm{mod }1)$$
    where $\alpha$ is any fixed irrational number, and let $\varphi(x)=x^{-1/2}\in L^1(\mathbb{T},\mathscr{B}_\mathbb{T},\mu)$ where $\mu=dx$. Then there is a summing sequence $\{S_n;n\in\mathbb{N}\}$ in $\mathbb{Z}$ so that
    \begin{gather*}
    \limsup_{n\to\infty}\frac{1}{|S_n|}\sum_{i\in S_n}\varphi(T_ix)=+\infty\quad \forall x\in\mathbb{T};
    \end{gather*}
    see \cite[Theorem~1]{Eme}. Now for any point $x_0\in\mathbb{T}$ we can find a F{\o}lner sequence $\{F_n^\prime; n\in\mathbb{N}\}$ in $\mathbb{Z}$ such that
   \begin{gather*}
    \lim_{n\to\infty}\frac{1}{|F_n^\prime|}\sum_{i\in F_n^\prime}\varphi(T_ix_0)=+\infty;
    \end{gather*}
    and then for any F{\o}lner subsequence $\{F_{n}\}$ of $\{F_n^\prime\}$, there is no a desired limit $\varphi^*$ as in Theorem~\ref{thm0.6}. In fact, Theorem~\ref{thm0.6} does not hold for $L^p(X,\mathscr{X},\mu)$ instead of $L^\infty(X,\mathscr{X})$, $p<\infty$, by \cite[Proposition~5]{Eme}.

\item[(2)] The question of \textit{pointwise} convergence of ergodic averages is much more delicate than \textit{mean} convergence. When $G$ is a \textit{second countable} amenable group, in 2001~\cite{Lin} Lindenstrauss proved the most general $L^1$-pointwise ergodic theorem that holds for $G$ acting on Lebesgue spaces, over any $\mathscr{K}$-F{\o}lner sequence $\{F_n; n\in\mathbb{N}\}$ satisfying the so-called \textbf{Shulman/tempered Condition}~\cite{Shu, Lin}:
    $\left|{\bigcup}_{k<n}F_k^{-1}F_n\right|\le C|F_n|$
    for some $C>0$ and all $n\ge1$. This condition is weaker than the so-called \textbf{Tempelman Condition} for summing sequence:
    $\left|F_n^{-1}F_n\right|\le C|F_n|$
    for some $C>0$ and all $n\ge1$, needed in \cite{Bew, Tem, Eme, OW} and so on.
    \begin{itemize}
    \item The first point of our Theorem~\ref{thm0.6} is over F{\o}lner sequence neither $\mathscr{K}$-F{\o}lner nor summing sequences. Although one can select out a tempered $\mathscr{K}$-F{\o}lner subsequence from any $\mathscr{K}$-F{\o}lner sequence (cf.~\cite[Proposition~1.4]{Lin}), yet we cannot select out a tempered $\mathscr{K}$-F{\o}lner subsequence from any F{\o}lner sequence that is not of $\mathscr{K}$-F{\o}lner.
    \item The other point is that our amenable group $G$ we consider here is not necessarily second countable. The second countability is an important condition for Lindenstrauss's basic covering result \cite[Lemma~2.1]{Lin} which is used in his proof of the $L^1$-pointwise ergodic theorem.
    \end{itemize}
    So our Theorem~\ref{thm0.6} is not a consequence of Lindenstrauss's $L^1$-pointwise ergodic theorem nor of his proving. In other words, there are other non-tempered F{\o}lner sequences which are admissible for $L^\infty$-pointwise convergence.

\item[(3)] A.~del~Junco and J.~Rosenblatt showed in \cite{JR} that on every nontrivial Lebesgue space $(X,\mathscr{B},\mu)$ one can find an ergodic $\mathbb{Z}$-action and some $\varphi\in L^\infty(X,\mathscr{B},\mu)$ such that $A(F_n^\prime,\varphi)(x)$ does not have a limit almost everywhere over the F{\o}lner sequence $\{F_n^\prime; n\in\mathbb{N}\}$ in $\mathbb{Z}$ defined by
    $F_n^\prime=\left\{n^2,n^2+1,\dotsc,n^2+n\right\}$ that satisfies Tempelman's condition but is not increasing.
    However, by our Theorem~\ref{thm0.6}, it follows that we can always find some F{\o}lner subsequence $\{F_{n}\}$ of $\{F_n^\prime\}$ such that
    \begin{gather*}
    A(F_n,\varphi)(x)\to \varphi^*(x)\quad \textit{a.e.}~(\mu).
    \end{gather*}
    This shows that the pointwise convergence is very sensitive to the choice of the F{\o}lner sequence in $G$.
\item[(4)] The continuity of the Moore-Smith limit $A(\centerdot)\colon L^\infty(X,\mathscr{X})\rightarrow\mathbb{R}^X;\ \varphi\mapsto\varphi^*$ in Theorem~\ref{thm0.5} is a new ingredient for non-ergodic systems; and it is another important point that $\varphi^*(x)$ is defined at everywhere $x$ in $X$ (not only a.e.) for any bounded observation $\varphi$ (This point will be needed in Lemma~\ref{lem0.7} below); moreover, $\varphi^*$ does not depend on any invariant measure $\mu$. These points are new observations in ergodic theory, since here $G\curvearrowright_TX$ is not assumed to be unique ergodic and even not topological.
\end{enumerate}
\end{remark-of-thm0.6}

We will provide a little later some applications of Theorem~\ref{thm0.5} and its proof idea in, respectively, $\S\ref{sec0.3.1}$ for probability theory, $\S\S\ref{sec0.3.2}, \ref{sec0.3.3}\textrm{ and }\ref{sec0.3.4}$ for ergodic theory, and $\S\ref{sec0.3.5}$ for the $L^\infty$-pointwise convergence of multiple ergodic averages of amenable module actions. Some further applications such as for product of amenable groups and quasi-weakly almost periodic points will be presented later in $\S\ref{sec3.5}$ and $\S\ref{sec4}$.

\subparagraph*{Outline of the proof of Theorems~\ref{thm0.5} and \ref{thm0.6}}
Since K.~Yosida and S.~Kakutani 1939~\cite{YK} the maximal ergodic theorem (or maximal inequality) has played an essential role in the proofs of the $L^p$-pointwise ergodic theorems (cf.~e.g., \cite{Tem, Shu, Lin, Nev, BN} etc.). In our situation, however, relative to a left F{\o}lner net $\{F_{\theta};\theta\in\Theta\}$ in $G$, yet there is no such a maximal inequality and even though $M^*\varphi(x):=\sup_{\theta\in\Theta}|A(F_\theta,\varphi)(x)|$, for $\varphi\in L^\infty(X,\mathscr{X})$, is not $\mathscr{X}$-measurable (cf.~\cite[$\S2.3.1$]{Nev}). Thus we here cannot expect to use the standard recipe for proving $L^p$-pointwise ergodic theorems for $1\le p<\infty$ (cf.~\cite[$\S2.4$]{Nev}).

In $\S\ref{sec3.2}$ not involving any ergodic theory, we will first prove the everywhere $L^\infty$-pointwise convergence, by only using the classical Arzel\'{a}-Ascoli theorem,
\begin{gather*}
A(F_{\theta},\varphi)(x)\to \varphi^*(x)\ \forall x\in X,\quad\forall \varphi\in L^\infty(X,\mathscr{X}),
\end{gather*}
for some F{\o}lner subnet $\{F_{\theta};\theta\in\Theta\}$ that relies on $G\curvearrowright_TX$ but not on the observation $\varphi(x)$.
To obtain a $L^\infty$-admissible F{\o}lner subnet in $G$ that is independent of an explicit Borel $G$-space $X$, we need to consider a ``big'' Borel $G$-space $\X=\prod_\lambda X_\lambda$ which is just the product space of all Borel $G$-spaces $X_\lambda$. See $\S\ref{sec3.3}$.

But, since $T_hT_g\varphi=T_{gh}\varphi\not=T_{hg}\varphi$ and $m_G$ is only a left Haar measure of $G$, we cannot deduce the $G$-invariance, and specially the measurability, of the Moore-Smith limit $\varphi^*$. To obtain the invariance and measurability, we will need to utilize in $\S\ref{sec3.3}$ the mean ergodic theorem (i.e. Theorem~\ref{thm0.4}) which will be proved in $\S\ref{sec2}$ using classical functional analysis (cf.~$\S\ref{sec1}$). In other words, the $L^p$-mean ergodic theorem will play a role in proving Theorem~\ref{thm0.6}.

Note that the mean ergodic theorem early played a role for a.e. $L^1$-pointwise convergence in \cite{Yos40, Yos80, Eme} in a different way based on Banach's convergence theorem. However, the latter is invalid in our situation because of having no the ``$L^\infty$-mean ergodic theorem'' condition
\begin{gather*}
\lim_{n\to\infty}\|A(F_n,\varphi)-\varphi^*\|_{\infty,\mu}=0
\end{gather*}
as the condition (4) of \cite[Theorem~XIII.2.1]{Yos80} and of $L^\infty(X,\mathscr{X})$ is never a subset of second category of $(L^1(X,\mathscr{X},\mu),\|\cdot\|_1)$. This means that even for Theorem~\ref{thm0.6}, the standard recipe for proving $L^p$-pointwise ergodic theorems for $1\le p<\infty$ is not valid here.
\subsection{Applications}
We will present four applications of Theorem~\ref{thm0.5} and its proof methods here. Our $L^\infty$-pointwise convergence at \textit{everywhere} is an important point for our arguments.

\subsubsection{Universal conditional expectations}\label{sec0.3.1}
We now present our first application of Theorems~\ref{thm0.5} and \ref{thm0.6} in probability theory and ergodic theory. As in Theorem~\ref{thm0.6}, let $G\curvearrowright_TX$ be a Borel action of an amenable group $G$ on a measurable space $(X,\mathscr{X})$.
Set
\begin{gather*}
\mathscr{X}_G=\left\{B\in\mathscr{X}\,|\,T_g^{-1}[B]=B\ \forall g\in G\right\}\intertext{and}
\mathscr{X}_{G,\mu}=\left\{B\in\mathscr{X}\,|\,\mu\left(T_g^{-1}[B]\vartriangle B\right)=0\ \forall g\in G\right\}\quad \forall \mu\in\mathcal{M}(G\curvearrowright_TX)
\end{gather*}
which both form $\sigma$-subfields of $\mathscr{X}$. Then $P(\phi)=E_\mu(\phi|\mathscr{X}_{G,\mu})$ for $\phi\in L^1(X,\mathscr{X},\mu)$ by Theorem~\ref{thm0.4}; moreover, function $\phi$ is $\mathscr{X}_G$-measurable if and only if $T_g\phi=\phi\; \forall g\in G$.

Now by Theorems~\ref{thm0.5} and~\ref{thm0.6}, we can obtain the following, which says that $\mathscr{X}_G$ is sufficient to the family $\mathcal{M}(G\curvearrowright_TX)$.

\begin{Lem}\label{lem0.7}
Let $G$ be $\sigma$-compact amenable and $X$ a Borel $G$-space with $\mathcal{M}(G\curvearrowright_TX)\not=\varnothing$. Then there exists a bounded linear operator
$\mathcal{E}\colon L^\infty(X,\mathscr{X})\rightarrow L^\infty(X,\mathscr{X})$ such that
\begin{enumerate}
\item[$(1)$] $\|\mathcal{E}(\centerdot)\|\le1$,
\item[$(2)$] $\mathcal{E}(\phi)=\phi$ if $\phi\in L^\infty(X,\mathscr{X}_G)$, and
\item[$(3)$] for any $\mu\in\mathcal{M}(G\curvearrowright_TX)$, $\mathcal{E}(\phi)$ is a version of $E_\mu(\phi|\mathscr{X}_{G,\mu})$ (i.e. $\mathcal{E}(\phi)=E_\mu(\phi|\mathscr{X}_{G,\mu})$ $\mu$-a.e.) for any $\phi\in L^\infty(X,\mathscr{X})$.
\end{enumerate}
\end{Lem}

\begin{proof}
Define $\mathcal{E}\colon \varphi\mapsto\varphi^*$ for each $\varphi\in L^\infty(X,\mathscr{X})$ where $\varphi^*$ is given as in Theorem~\ref{thm0.6} based on some $L^\infty$-admissible F{\o}lner sequence, say $\{F_n; n=1,2,\dotsc\}$ for $G$. Then (1) and (2) of Lemma~\ref{lem0.7} both hold automatically.

To check (3), let $\mu\in\mathcal{M}(G\curvearrowright_TX)$ and $B\in\mathscr{X}_{G,\mu}$ be arbitrary. Then by Theorem~\ref{thm0.6} and Fubini's theorem follows that for $\phi\in L^\infty(X,\mathscr{X})$,
\begin{equation*}\begin{split}
\int_B\mathcal{E}(\phi)d\mu
&=\lim_{n\to\infty}\frac{1}{|F_n|}\int_{F_n}\int_X1_B(x)T_g\phi(x)d\mu(x)dg\\
&=\lim_{n\to\infty}\frac{1}{|F_n|}\int_{F_n}\int_X1_B(T_gx)\phi(T_gx)d\mu(x)dg\\
&=\lim_{n\to\infty}\frac{1}{|F_n|}\int_{F_n}\int_B\phi d\mu dg\\
&=\int_B\phi d\mu
\end{split}\end{equation*}
which proves (3) since $\phi^*=P(\phi)$ \textit{a.e.} is $\mathscr{X}_{G,\mu}$-measurable by Theorem~\ref{thm0.6}.
\end{proof}

In 1963 V.\;S.~Varadarajan showed a universal conditional expectation theorem for any locally compact \textit{second countable} Hausdorff (abbreviated to lcscH) group $G$ by using harmonic analysis combining with the martingale theory~\cite[Theorem~4.1]{Var}. Here our ergodic-theoretic proof is very concise by using sufficiently the \textit{$\sigma$-compact amenability} without assuming the second countability axiom.

On the other hand, since the $L^1$-pointwise ergodic theorem of Lindenstrauss is for lcscH group and is for \textit{a.e.} convergence, not for \textit{everywhere} convergence as in our Theorem~\ref{thm0.6}, hence the classical $L^p$-pointwise ergodic theorems in the literature cannot play a role in the proof of Lemma~\ref{lem0.7} in place of Theorem~\ref{thm0.6}.

\subsubsection{Universal ergodic disintegration of invariant measures}\label{sec0.3.2}
We now turn to the important decomposition of an invariant probability measure into ergodic components. Let's consider a Borel $G$-space $X$, where $G$ is an lcH group and $X$ a compact Hausdorff space with $\mathscr{X}=\mathscr{B}_X$ the Borel $\sigma$-field. As usual, $\mu\in\mathcal{M}(G\curvearrowright X)$ is said to be \textbf{\textit{ergodic}} if and only if the \textit{0-1 law} holds:
\begin{itemize}
\item $\mu(B)=0$ or $1$ $\forall B\in\mathscr{X}_{G,\mu}$.\footnote{It should be noted here that sometimes $\mu$ is called \textit{ergodic} if weakly $\mu(B)=0$ or $1$ $\forall B\in\mathscr{X}_{G}$. If $G$ is lcscH, then the two cases are equivalent to each other; see, e.g., \cite{Var, Far, Phe}. However, in general, the latter is weaker than the former for $\mathscr{X}_G\subsetneq\mathscr{X}_{G,\mu}$ in general; see \cite{Far} and \cite[\S12]{Phe} for a counterexample.}
\end{itemize}

The classical Ergodic Decomposition Theorem of Farrell-Varadarajan \cite{Far, Var} claims that \textit{if $X$ is ($G$-isomorphically) compact metrizable and $G$ is lcscH, then  every $\mu\in\mathcal{M}(G\curvearrowright X)$ has a.s. uniquely an ergodic disintegration $\beta\colon x\mapsto\mu_x$}. See also \cite[Theorem~8.20]{EW}.

However, if $G$ is not second countable, then the methods developed in \cite{Far, Var, EW} are all invalid. Moreover, if $(X,\mathscr{X},\mu)$ is not (isomorphically) a compact metrizable probability space, then there exists no a disintegration of $\mu$ in general; see, e.g., \cite{Die} for counterexamples. Our Theorem~\ref{thm0.8} below shows that a universal ergodic disintegration (good for all $\mu\in\mathcal{M}(G\curvearrowright X)$) is always existent if $X$ is $G$-isomorphically a compact metric space.

Let $\mathscr{M}(X)$ consist of all the Borel probability measures on $X$ endowed with the customary topology as follows:
\begin{itemize}
\item the function $\mu\mapsto\mu(f)=\int_Xfd\mu$ is continuous, for any fixed $f\in C(X)$.
\end{itemize}
Now based on Theorem~\ref{thm0.6} and Lemma~\ref{lem0.7}, we can then concisely obtain a universal ergodic disintegration in a very general situation.

\begin{Thm}\label{thm0.8}
Let $G\curvearrowright_TX$ be a Borel action of a $\sigma$-compact amenable group $G$ by continuous transformations of a compact metric space $X$ to itself such that $\mathcal{M}(G\curvearrowright_TX)\not=\varnothing$. Then there exists an $\mathscr{X}$-measurable mapping
\begin{gather*}
\beta\colon X\rightarrow\mathcal{M}(G\curvearrowright_TX);\ x\mapsto\beta_x
\end{gather*}
such that for any $\mu\in\mathcal{M}(G\curvearrowright_TX)$,
\begin{enumerate}
\item[$(1)$] for any $\phi\in L^1(X,\mathscr{X},\mu)$, it holds that
\begin{enumerate}
\item $\phi\in L^1(X,\mathscr{X},\beta_x)$ and $\int_X\phi d\beta_x=E_\mu\left(\phi|\mathscr{X}_{G,\mu}\right)(x)$ for $\mu$-a.e. $x\in X$;
\item $\mu=\int_X\beta_xd\mu(x)$, i.e. $\int_X\varphi d\mu=\int_X\int_X\varphi d\beta_xd\mu(x)\ \forall \varphi\in L^1(X,\mathscr{X},\mu)$;
\end{enumerate}
\item[$(2)$] moreover, $\beta_x$ satisfies the 0-1 law (i.e. $\beta_x$ is ergodic to $G\curvearrowright_TX$) for $\mu$-a.e. $x\in X$.
\end{enumerate}
Namely $\{\beta_x; x\in X\}$, characterized by $(1)$, is a ``universal ergodic disintegration'' of $G\curvearrowright_TX$.
\end{Thm}

\begin{proof}
For the convenience of our later arguments in $\S\ref{sec0.3.4}$ we will illustrate our proof by dividing into three steps.

\textit{Step 1. A universal construction.}
Let $\varphi\mapsto\varphi^*$ from $L^{\infty}(X,\mathscr{X})$ to $\mathbb{R}^X$ be as in Theorem~\ref{thm0.6} associated to some $L^\infty$-admissible F{\o}lner sequence $\{F_n\}_1^\infty$ for $G$. Then for any $x\in X$, we can define a positive linear functional $L_x\colon C(X)\rightarrow\mathbb{R}$ by $\varphi\mapsto \varphi^*(x)$
such that $L_x(\mathbf{1})=1$. By the Riesz representation theorem, it follows that for any $x\in X$, there exists a (unique) probability measure, write $\beta_x$, on $\mathscr{X}$ $(=\mathscr{B}_X)$ such that
$$L_x(\varphi)=\int_X\varphi d\beta_x\quad \forall \varphi\in C(X),$$
and $\beta_x$ is independent of any $\mu$ in $\mathcal{M}(G\curvearrowright_TX)$.
Given any $\varphi\in C(X)$, clearly $x\mapsto\beta_x(\varphi)=\varphi^*(x)$ is $\mathscr{X}$-measurable by Theorem~\ref{thm0.6}. Hence $X\ni x\mapsto\beta_x\in\mathscr{M}(X)$ is $\mathscr{X}$-measurable.

In order to check the $G$-invariance of $\beta_x$, for any $h\in G$ and $\varphi\in C(X)$, we can see $T_h\varphi\in C(X)$ and that
\begin{align*}
\int_X\varphi d\beta_x&=\lim_{n\to\infty}\frac{1}{|F_n|}\int_{hF_n}T_g\varphi(x)dg
=\lim_{n\to\infty}\frac{1}{|F_n|}\int_{F_n}T_{hg}\varphi(x)dg\\
&=\lim_{n\to\infty}\frac{1}{|F_n|}\int_{F_n}T_{g}T_h\varphi(x)dg\\
&=\int_XT_h\varphi d\beta_x
\end{align*}
which implies that $T_h\beta_x=\beta_x$. Thus, $\beta_x\in\mathcal{M}(G\curvearrowright_TX)$. (This also proves the classical result of M.~Day: $\mathcal{M}(G\curvearrowright_TX)\not=\varnothing$.)

\textit{Step 2. Checking property (1).}
First the property (b) is valid whenever (a) is true.
Noting the property (a) holds, for any $\phi\in C(X)$, by Lemma~\ref{lem0.7}. The general $L^1$-case follows from the standard approximation theorem and the Lebesgue dominated convergence theorem of conditional expectation (cf., e.g.,~\cite[pp.~109]{Fur}). Indeed, for $\phi\in L^1(X,\mathscr{X},\mu)$ note that $\phi=\sum_{n}\phi_n$, $\mu$-\textit{a.e.}, $\phi_n\in C(X)$ where $\sum_n\|\phi_n\|_1<\infty$.

Set $A=\{x\colon \phi(x)\not=\sum_n\phi_n(x)\}$. Then $\mu(A)=0$ and let $A_n$ be open neighborhoods of $A$ with $A_n\supset A_{n+1}$ and $\mu(A_n)\to0$. Let $\chi_n\in C(X)$ such that $\chi_n=0$ outside $A_n$ and $0<\chi_n\le 1$ on $A_n$. For each $\varepsilon>0$, the power $\chi_n^\varepsilon$ is continuous and
$$
\int_X\left(\int\chi_n^\varepsilon d\beta_x\right)d\mu(x)=\int_X\chi_n^\varepsilon d\mu\le\mu(A_n).
$$
Letting first $\varepsilon\to0$ and then $n\to\infty$, we find
\begin{gather*}
0=\lim_{n\to\infty}\lim_{\varepsilon\to0}\int_X\left(\int\chi_n^\varepsilon d\beta_x\right)d\mu(x)=\lim_{n\to\infty}\int_X\beta_x(A_n)d\mu(x)=\int_X\beta_x({\cap}_nA_n)d\mu(x).
\end{gather*}
Hence $\beta_x(A)=0$ for $\mu$-\textit{a.e.} $x\in X$. This implies that $\phi=\sum_n\phi_n$, $\beta_x$-\textit{a.e.}, for $\mu$-\textit{a.e.} $x\in X$.

Next, $\sum_n\|E(|\phi_n||\mathscr{X}_{G,\mu})\|_1<\infty$, so that $X_0:=\left\{x\colon \sum_nE(|\phi_n||\mathscr{X}_{G,\mu})(x)\ \left(=\sum_n\beta_x(|\phi_n|)\right)<\infty\right\}$ is of $\mu$-measure one. Now for $x\in X_0$, $\sum_n|\phi_n|\in L^1(X,\mathscr{X},\beta_x)$, and so
$$
\int_X\left(\sum_n\phi_n\right)d\beta_x=\sum_n\int_X\phi_nd\beta_x.
$$
We also have $\mu$-\textit{a.e.} $x\in X$, $\int_X\phi_nd\beta_x=E(\phi_n|\mathscr{X}_{G,\mu})(x)$, so that
$$
\int_X\left(\sum_n\phi_n\right)d\beta_x=\sum_nE(\phi_n|\mathscr{X}_{G,\mu})(x)\quad (\mu\textit{-a.e. }x\in X).
$$
Since $\phi=\sum_n\phi_n$ ($\beta_x$-\textit{a.e.}), we find
$$
\int_X\phi d\beta_x=\sum_nE(\phi_n|\mathscr{X}_{G,\mu})(x)=E\left(\sum_n\phi_n|\mathscr{X}_{G,\mu}\right)(x)=E(\phi|\mathscr{X}_{G,\mu})(x)\quad (\mu\textit{-a.e. }x\in X)
$$
This proves the property (b).

\textit{Step 3. Checking property (2).}
Finally, we will proceed to prove the property (2). For this, let $\mu\in\mathcal{M}(G\curvearrowright_TX)$ be any given. By the property (1), it follows that
for any $B\in\mathscr{X}_{G,\mu}$,
\begin{gather*}
\beta_x(B)=E_\mu\left(1_B|\mathscr{X}_{G,\mu}\right)(x)=1_B(x)\quad (\mu\textit{-a.e. }x\in X).
\end{gather*}
That is to say, for $B\in\mathscr{X}_{G,\mu}$, we have $1_B\equiv\beta_x(B)$ $\beta_x$-\textit{a.e.}, for $\mu$-\textit{a.e.} $x\in X$.
Thus, if $\varphi\in L^\infty(X,\mathscr{X})$ is $\mathscr{X}_{G,\mu}$-measurable, then
\begin{gather*}
\varphi\equiv\int_X\varphi d\beta_x,\quad \beta_x\textit{-a.e.},\quad (\mu\textit{-a.e.}~x\in X).
\end{gather*}
Let $\mathcal{B}=\{B_1,B_2,\dotsc\}$ be a countable algebra of $\mathscr{X}$-subsets of $X$ with $\sigma(\mathcal{B})=\mathscr{X}$ and $\mathcal{E}$ be as in Lemma~\ref{lem0.7}. Then one can find some $\mathscr{E}_\mu\in\mathscr{X}$ with $\mu(\mathscr{E}_\mu)=1$ such that for all $x\in\mathscr{E}_\mu$ and $B\in\mathcal{B}$,
\begin{gather*}
\mathcal{E}(1_B)\stackrel{\beta_x\textit{-a.e.}}{\equiv}\int_X\mathcal{E}(1_B)d\beta_x=\int_XE_{\beta_x}\left(1_B\big{|}\mathscr{X}_{G,\beta_x}\right)d\beta_x=\beta_x(B).
\end{gather*}
Now for any $x\in\mathscr{E}_\mu$ and any $I\in\mathscr{X}_{G,\beta_x}$, choose a sequence $B_{n_k}\in\mathcal{B}$ with $1_{B_{n_k}}\to 1_I$ a.e.~($\beta_x$); and then
\begin{align*}
\beta_x(I)&=\lim_{k\to\infty}\beta_x(B_{n_k})\stackrel{\beta_x\textit{-a.e.}}{=}\lim_{k\to\infty}E_{\beta_x}\left(1_{B_{n_k}}\big{|}\mathscr{X}_{G,\beta_x}\right)\stackrel{\beta_x\textit{-a.e.}}{=}E_{\beta_x}\left(\lim_{k\to\infty}1_{B_{n_k}}\big{|}\mathscr{X}_{G,\beta_x}\right)\\
&=E_{\beta_x}\left(1_I\big{|}\mathscr{X}_{G,\beta_x}\right)\quad \beta_x\textit{-a.e.}\\
&=1_I\quad \beta_x\textit{-a.e.}
\end{align*}
by Lemma~\ref{lem0.7} and $\beta_x\in\mathcal{M}(G\curvearrowright_TX)$.
So $\beta_x(I)=0$ or $1$ and thus $\beta_x$ is ergodic for any $x\in\mathscr{E}_\mu$.

This therefore proves Theorem~\ref{thm0.8}.
\end{proof}

In Theorem~\ref{thm0.8} $G$ is not assumed commutative. The $\sigma$-compactness is needed in Theorem~\ref{thm0.8} to guarantee the measurability of $\beta_x$ and Lemma~\ref{lem0.7}. We now conclude $\S\ref{sec0.3.2}$ with some remarks on Theorem~\ref{thm0.8}.

\begin{remark-of-thm0.8}
\begin{enumerate}
\item[(1)] In Theorem~\ref{thm0.8}, our disintegration $\{\beta_x,x\in X\}$ is independent of $\mu$ and it is in fact \textit{a.s.} unique.

\item[(2)] Each invariant component $\beta_x$ is constructed over a same $L^\infty$-admissible F{\o}lner sequence $\{F_n\}_1^\infty$ for $G$. It should be noted that disintegration $\{\beta_x;x\in X\}$ is important in many aspects of ergodic theory; for example, for defining fiber product $\int_Y\mu_y\otimes\mu_yd\nu$ of a $G$-factor $\pi\colon(X,\mathscr{X},\mu)\rightarrow(Y,\mathscr{Y},\nu)$ with $\pi^{-1}[\mathscr{Y}]=\mathscr{X}_{G,\mu}$ in Furstenberg's theory~\cite{Fur}.
\item[(3)] If $G$ acts by continuous maps on a compact metric space $X$, then one can obtain an ergodic integral representation for any $\mu\in\mathcal{M}(G\curvearrowright X)$ by Choquet's theorem (cf.~\cite[$\S12$]{Phe}); but no the disintegration $\{\beta_x;x\in X\}$.
\item[(4)] Let $G$ be a $\sigma$-compact amenable group, which is \textit{not} second countable. Then the Farrell-Varadarajan theorem plays no role in this case. Moreover, since $G$ is not metrizable (otherwise it is second countable), hence \cite[Theorem~8.20]{EW} cannot play a role too.
\item[(5)] If $G$ is not amenable, then Theorem~\ref{thm0.8} does not need to be true. For example, for $G$ we take the discrete group of \textit{all} permutations of $\mathbb{N}$ and we naturally act it on the compact metric space $\{0,1\}^\mathbb{N}$ (cf.~\cite{Fom,Var}).
\end{enumerate}
\end{remark-of-thm0.8}
\subsubsection{A Khintchine-type recurrence theorem}\label{sec0.3.3}
Let $\mathbb{R}$ be equipped with the standard Euclidean topology. Then A.\,Y.~Khintchine's recurrence theorem says that (cf.~\cite{Khi, NS}):
\begin{itemize}
\item If $\mathbb{R}\curvearrowright_TX$ is a continuous action of $\mathbb{R}$ on a compact metric space $X$ preserving a Borel probability measure $\mu$, then for any $E\in\mathscr{B}_X$ with $\mu(E)>0$, the set
    \begin{gather*}
    \left\{t\in\mathbb{R}\,|\,\mu(E\cap T_{-t}E)>\mu(E)^2-\varepsilon\right\},\quad \forall \varepsilon>0,
    \end{gather*}
    is relatively dense in $\mathbb{R}$
\end{itemize}
V.~Bergelson, B.~Host and B.~Kra have recently derived in 2005 \cite{BHK} the following multiple version of Khintchine's theorem, only valid in \textit{ergodic} $\mathbb{Z}$-action $\mathbb{Z}\curvearrowright_TX$:
\begin{itemize}
\item $\{n\in\mathbb{Z}\,|\,\mu(E\cap T_{-n}E\cap T_{-2n}E)>\mu(E)^3-\varepsilon\}$ and \\
$\{n\in\mathbb{Z}\,|\,\mu(E\cap T_{-n}E\cap T_{-2n}E\cap T_{-3n}E)>\mu(E)^4-\varepsilon\}$,\\
each of the above two sets is relatively dense in $\mathbb{Z}$.
\end{itemize}
Using Theorems~\ref{thm0.4}, \ref{thm0.6} and \ref{thm0.8} and differently with \cite{Khi, NS}, we may now generalize Khintchine's theorem from $\mathbb{R}$-actions to amenable-group actions as follows:

\begin{Thm}\label{thm0.9}
Let $G\curvearrowright_TX$ be a $\mu$-preserving Borel action of a $\sigma$-compact amenable group $G$, by continuous transformations on a compact metric space $X$, not necessarily ergodic. Then for any $\varphi\in L^2(X,\mathscr{B}_X,\mu)$ with $\varphi\ge0$ a.e. and $\int_X\varphi d\mu>0$, the set
\begin{gather*}
\mathcal{H}(\varphi,\varepsilon)=\left\{g\in G\,\bigg{|}\,\int_X\varphi(x)\varphi(T_gx)d\mu(x)>\left(\int_X\varphi d\mu\right)^2-\varepsilon\right\},\quad \forall \varepsilon>0,
\end{gather*}
is of positive lower Banach density; i.e., $\mathcal{H}(\varphi,\varepsilon)$ has positive lower density over any F{\o}lner sequence $\{F_n\}_1^\infty$ in $G$. This implies that $\mathcal{H}(\varphi,\varepsilon)$ is syndetic in $G$.
\end{Thm}

\begin{proof}
Let $\{F_n\}_1^\infty$ be any F{\o}lner sequence in $G$.
Let $\{\beta_x,x\in X\}$ be the universal ergodic disintegration of $G\curvearrowright_TX$ by Theorem~\ref{thm0.8}. Then for $\mu$\textit{-a.e. }$x_0\in X$, by using Theorem~\ref{thm0.6} for $\beta_{x_0}$ in place of $\mu$,
\begin{equation*}
\lim_{n\to\infty}\frac{1}{|F_n|}\int_{F_n}T_g\varphi(x)dg=\varphi^*(x)=\int_X\varphi d\beta_{x_0}\quad \left(\textrm{in }(L^2(X,\mathscr{X},\beta_{x_0}),\|\centerdot\|_2)\right).
\end{equation*}
Furthermore,
\begin{gather*}
\lim_{n\to\infty}\frac{1}{|F_n|}\int_{F_n}\left(\int_X\varphi(x)\varphi(T_gx)d\beta_{x_0}(x)\right)dg=\left(\int_X\varphi d\beta_{x_0}\right)^2.
\end{gather*}
Therefore by Theorem~\ref{thm0.8}, Fatou lemma, Fubini's theorem and Jensen's inequality, we can obtain that
\begin{equation*}\begin{split}
\liminf_{n\to\infty}\frac{1}{|F_n|}\int_{F_n}\left(\int_X\varphi(x)\varphi(T_gx)d\mu(x)\right)dg&\ge\int_X\left(\int_X\varphi d\beta_{x_0}\right)^2d\mu(x_0)\\
&\ge\left(\int_X\int_X\varphi d\beta_{x_0}d\mu(x_0)\right)^2\\
&=\left(\int_X\varphi d\mu\right)^2.
\end{split}\end{equation*}
This means that for any $\varepsilon>0$, the set $\mathcal{H}(\varphi,\varepsilon)$ has positive lower density over $\{F_n\}_1^\infty$ in $G$.

The proof of Theorem~\ref{thm0.9} is thus completed, for $\{F_n\}_1^\infty$ is arbitrary.
\end{proof}

A special important case is that $\varphi(x)=1_B(x)$ for $B\in\mathscr{B}_X$ with $\mu(B)>0$. Moreover, since here $G$ is not necessarily an abelian group and so there is no an applicable analogue
of Herglotz-Bochner spectral theorem (cf.~\cite[Theorem~36A]{Loo} for locally compact abelian group) for the correlation function $g\mapsto\langle \varphi, T_g\varphi\rangle$ from $G$ to $\mathbb{C}$, our proof of Theorem~\ref{thm0.9} is of somewhat interest itself.

We will consider the recurrence theorem in the case that $G$ is not $\sigma$-compact in $\S\ref{sec4.2}$.
\subsubsection{Existence of $\sigma$-finite invariant measures}\label{sec0.3.4}
Replacing $C(X)$ by $C_c(X)=\{\phi\colon X\rightarrow\mathbb{R}\,|\,\phi\textrm{ continuous with }\textrm{supp}(f)\textrm{ compact}\}$ in the above proof of Theorem~\ref{thm0.8}, we in fact have proved the following byproduct, which generalizes the classical theorems of Kryloff-Bogoliouboff and of Day.

\begin{prop}[$\sigma$-finite invariant measure]\label{pro0.10}
Let $G\curvearrowright_TX$ be a Borel action of an amenable group $G$ by continuous maps of a ``locally compact Hausdorff space'' $X$ to itself such that
\begin{itemize}
\item $\varphi^*(x)\not\equiv0$\quad for some $\varphi\in C_c(X)$ over some F{\o}lner net $\{F_\theta; \theta\in\Theta\}$ in $G$.
\end{itemize}
Then there exists a nontrivial Borel measure $\mu$ on $X$ with the following properties:
\begin{enumerate}
\item[$\mathrm{(i)}$] $\mu(K)<\infty$ for every compact set $K\subseteq X$;
\item[$\mathrm{(ii)}$] $\mu(E)=\inf\{\mu(V)\colon E\subseteq V, V\textrm{ open}\}$ for every $E\in\mathscr{B}_X$;
\item[$\mathrm{(iii)}$]$\mu(E)=\sup\{\mu(K)\colon K\subseteq E, K\textrm{ compact}\}$ for every open set $E$ and for each $E\in\mathscr{B}_X$ with $\mu(E)<\infty$; and
\item[$\mathrm{(iv)}$] for every $g\in G$, $\int_X\phi d\mu=\int_XT_g\phi d\mu$ for each $\phi\in C_c(X)$.
\end{enumerate}
In particular, if $X$ is a locally compact $\sigma$-compact metric space, then $\mu$ is a $\sigma$-finite invariant measure for $G\curvearrowright_TX$.
\end{prop}

\begin{proof}
As in Step 1 of the proof of Theorem~\ref{thm0.8}, it is easy to see that there exists a nontrivial Borel measure $\mu$ on $X$ with the properties (i)--(iv).

Now let $X$ be an lc, $\sigma$-compact metric space. We first need to prove that $\mu$ is quasi-invariant for $G\curvearrowright_TX$:
\begin{itemize}
\item $\forall E\in\mathscr{B}_X$ and $g\in G$, $\mu(E)>0\Leftrightarrow\mu\left(T_g^{-1}E\right)>0$.
\end{itemize}
If $\mu(E)>0$; then one can find some compact set $K\subseteq E, \mu(K)>0$ and $\phi_n\in C_c(X)$ with $\phi_n=1$ on $K$, $0\le\phi_n\le 1$ and $\phi_n(x)\searrow 1_K(x)$. By $\phi_n(T_gx)\to1_K(T_gx)$ and (iv), it follows that
\begin{gather}\label{eq0.4}
\mu\left(T_g^{-1}E\right)\ge\int_X1_K(T_gx)d\mu=\lim_{n\to\infty}\int_X\phi_n(T_gx)d\mu=\lim_{n\to\infty}\int_X\phi_nd\mu\ge\mu(K)>0.
\end{gather}
Conversely, if $\mu\left(T_g^{-1}E\right)>0$, then by the above argument with $E^\prime:=T_g^{-1}E$ in place of $E$ and $h:=g^{-1}$ in place of $g$, we can see that $0<\mu\left(T_{h}^{-1}E^\prime\right)=\mu(E)$. Thus $\mu$ is quasi-invariant for $G\curvearrowright_TX$.

In fact, (\ref{eq0.4}) also implies that $\mu\left(T_g^{-1}E\right)\ge\mu(E)$ for any $g\in G$. Thus, $\mu$ is invariant. This proves Proposition~\ref{pro0.10}.
\end{proof}

Clearly, Day's fixed-point theorem and Kryloff-Bogoliouboff's construction proofs, for any amenable group $G$ acting continuously on a compact metric space $X$, both does not work for proving the existence of a $\sigma$-finite $\infty$-invariant measure here; this is because the convex set of probability measures on $X$ is not compact for $X$ is not compact.

In fact, we may read Proposition~\ref{pro0.10} as follows:
\begin{itemize}
\item \textit{Let $G\curvearrowright_TX$ be a Borel action of an amenable group $G$, by continuous maps, on a locally compact $\sigma$-compact metric space $X$.
 If $G\curvearrowright_TX$ does not have any $\sigma$-finite invariant measure, then the Moore-Smith limit $\varphi^*\equiv0$ for any $\varphi\in C_c(X)$ over any $L^\infty$-admissible F{\o}lner net $\{F_\theta; \theta\in\Theta\}$ in $G$; i.e., $G\curvearrowright_TX$ is \textbf{dissipative}.}
\end{itemize}
Clearly, this proposition has some interests of its own.

\subsubsection{$L^\infty$-Pointwise multi-ergodic theorem for topological modules}\label{sec0.3.5}
In what follows, $\mathbb{G}$ is called a \textbf{\textit{topological $R$-module}} if and only if $\mathbb{G}$ is a module over a ring $(R,+,\cdot)$ such that
\begin{itemize}
\item $\mathbb{G}$ is a topological group, $(R,+,\cdot)$ is such that $(R,+)$ is an lcH abelian group with a fixed Haar measure $|\cdot|$ or $dt$, and the scalar multiplication $(t,g)\mapsto tg$ from $R\times \mathbb{G}$ to $\mathbb{G}$ is continuous.
\end{itemize}
We now consider a Borel action of a topological $R$-module $\mathbb{G}$ on a measurable space $(X,\mathscr{X})$, $T\colon \mathbb{G}\times X\rightarrow X$ or write $\mathbb{G}\curvearrowright_TX$. Given any $g\in\mathbb{G}$ and $t\in R$, we customarily write
\begin{gather*}
T_{g}^t\colon X\rightarrow X;\quad x\mapsto T_g^tx=T_{tg}x\ \forall x\in X.
\end{gather*}
Now for any $g_1,\dotsc,g_l\in \mathbb{G}$ and any compacta $K$ of $R$ with $|K|>0$, for any $\varphi\in L^\infty(X,\mathscr{X})$, we may define the \textbf{\textit{multiple ergodic average}} of $\varphi$ over $(g_1,\dotsc,g_l;K)$ as follows:
\begin{align*}
A_{g_1,\dotsc,g_l}(K,\varphi)(x)&=\frac{1}{|K|}\int_KT_{g_1}^t\varphi(x)\dotsm T_{g_l}^t\varphi(x)dt\quad \forall x\in X\\
&=\frac{1}{|K|}\int_K\varphi(T_{tg_1}x)\dotsm\varphi( T_{tg_l}x)dt.
\end{align*}
Our methods developed for Theorem~\ref{thm0.5} is also valid for the following multiple ergodic theorem.

\begin{Thm}\label{thm0.11}
Let $\mathbb{G}$ be a topological $R$-module; then we can select out a subnet $\{K_{\theta};\theta\in\Theta\}$ from any net $\{K_{\theta^\prime}^\prime;\theta^\prime\in\Theta^\prime\}$ of positive Haar-measure compacta of $(R,+)$ such that, if $\mathbb{G}\curvearrowright_TX$ is a Borel action of $\mathbb{G}$ on $X$, then for any $g_1,\dotsc,g_l\in \mathbb{G}$ and any $\varphi\in L^\infty(X,\mathscr{X})$,
\begin{gather*}
\lim_{\theta\in\Theta}A_{g_1,\dotsc,g_l}(K_{\theta},\varphi)(x)=A_{g_1,\dotsc,g_l}(\varphi)(x)\quad \forall x\in X,
\end{gather*}
where $A_{g_1,\dotsc,g_l}(\centerdot)\colon(L^\infty(X,\mathscr{X}),\|\cdot\|_\infty)\rightarrow(\mathbb{R}^X,\mathfrak{T}_{\mathrm{p}\textrm{-}\mathrm{c}})$ is continuous. Moreover, if $(R,+)$ is $\sigma$-compact and $(X,\mathscr{X})$ is countably generated, then for any $\varphi\in L^\infty(X,\mathscr{X})$,
\begin{gather*}
A_{g_1,\dotsc,g_l}(K_{\theta},\varphi)\xrightarrow[]{\textrm{weakly}}A_{g_1,\dotsc,g_l}(\varphi)\quad \textrm{in }L^2(X,\mathscr{X},\mu)
\end{gather*}
for any $\mu\in\mathcal{M}(\mathbb{G}\curvearrowright_TX)$ if $\mathcal{M}(\mathbb{G}\curvearrowright_TX)\not=\varnothing$.
\end{Thm}

\begin{proof}
This follows easily from a slight modification of the proof of Theorem~\ref{thm3.3} presented in $\S\ref{sec3.2}$ by noting
\begin{gather}
\begin{align*}
\left|A_{g_1,\dotsc,g_l}(K,\varphi)(x)-A_{g_1,\dotsc,g_l}(K,\psi)(x)\right|
&\le\frac{1}{|K|}\int_K\left|\varphi(T_{g_1}^tx)\dotsm\varphi(T_{g_l}^tx)-\psi(T_{g_1}^tx)\dotsm\psi(T_{g_l}^tx)\right|dt\\
&\le\sum_{k=0}^{l-1}\|\varphi-\psi\|_\infty\|\varphi\|_\infty^{l-1-k}\|\psi\|_\infty^k
\end{align*}
\end{gather}
in place of (\ref{eq3.2}).
So we omit the details.
\end{proof}

The interesting point of Theorem~\ref{thm0.11} lies in that $\{K_{\theta},\theta\in\Theta\}$ is independent of any observations $\varphi$ in $L^\infty(X,\mathscr{X})$.

However, we have lost the independence (with subnets) and invariance of $A_{g_1,\dotsc,g_l}(\varphi)$ in general. Of course, if we are only concerned with the classical case that $R=\mathbb{Z}$ and
\begin{gather*}
T_{g_1}^n=T^n, T_{g_2}^n=T^{2n}, \dotsc, T_{g_l}^n=T^{ln}
\end{gather*}
for a $\mu$-preserving map $T$ or
\begin{gather*}
T_{g_1}^n=T_1^n, T_{g_2}^n=T_2^{2n}, \dotsc, T_{g_l}^n=T_l^{ln}
\end{gather*}
for $\mu$-preserving commuting maps $T_1,T_2,\dotsc,T_l$, then combining with the multiple $L^2$-mean ergodic theorems of Host-Kra \cite{HK} (also Ziegler~\cite{Zie}) or of Tao~\cite{Tao} (also Austin~\cite{Aus}) we can see
$A_{T,T^2,\dotsc,T^{l}}(\varphi)$ or $A_{T_1,T_2,\dotsc,T_{l}}(\varphi)$ is $\mathscr{X}$-measurable and independent of the choice of F{\o}lner subsequence.

\begin{cor}
Let $T_1,\dotsc,T_l$ be $l$ commuting $\mu$-preserving automorphism of $X$. Then one can select out a F{\o}lner subsequence $\{F_{n}\}$ from any F{\o}lner sequence $\{F_{n^\prime}^\prime\}$ in $(\mathbb{N},+)$ such that for any $\varphi\in L^\infty(X,\mathscr{X})$,
\begin{gather*}
\lim_{n\to\infty}\frac{1}{|F_{n}|}\sum_{i\in F_{n}}\varphi(T_1^ix)\dotsm\varphi(T_l^ix)=A_{T_1,\dotsc,T_l}(\varphi)(x)\quad\mu\textrm{-a.e.}
\end{gather*}
with $A_{T_1,\dotsc,T_l}(\varphi)\in L^\infty(X,\mathscr{X},\mu)$ is independent of the choice of the subsequence $\{F_{n}\}$.
\end{cor}

We note that if the $L^\infty$-pointwise convergence is required for the previously given F{\o}lner net itself, not over its F{\o}lner subnet, then the question is much harder and more delicate; see, e.g., \cite{Bou} for $2$-tuple commuting case and \cite{HSY} for ergodic distal systems.

\begin{que}
In Theorem~\ref{thm0.11}, if we consider the sequential limit over a F{\o}lner sequence in a $\sigma$-compact $(R,+)$, then it is possible that $A_{g_1,\dotsc,g_l}(\varphi)$ is a.e. independent of the choice of the $L^\infty$-admissible F{\o}lner subsequences.
\end{que}
\section{The adjoint operators, null spaces and range spaces}\label{sec1}

Let $E$ be a Banach space and let $\mathscr{L}(E)$ denote the space of all continuous linear operators of $E$ into itself. By $E^*$ we denote the dual Banach space of $E$, which consists of all the bounded linear functionals of $E$. For $x\in E$ and $\phi\in E^*$, write $\phi(x)=\langle x,\phi\rangle$ if no confusion. Recall that $E$ is called \textbf{\textit{reflexive}} if and only if $E=E^{**}$.
In this section we will prove a preliminary theorem (Theorem~\ref{thm1.2} below).

\subsection{Orthogonal complements, projection and adjoint operators}
Let $H$ be a subset of $E$. Then we define
$$H^\perp=\{\phi\in E^*\,|\,\langle x,\phi\rangle=0\ \forall x\in H\}.$$
Clearly, $H^\perp=\overline{H}^\perp$ is a closed linear subspace of $E^*$. A continuous linear operator $P$ of $E$ into itself satisfying $P^2=P$ is called a \textbf{\textit{projection}} of $E$ to itself (cf.~\cite[Definition~6.1]{Lor}).

Given any $T\in\mathscr{L}(E)$, we define a continuous linear operator of $E^*$,
$T^*\colon E^*\rightarrow E^*$, by $\phi\mapsto\phi\circ T$,
which is called the \textit{adjoint} to $T$, such that
$$\|T^*\|=\|T\|\ \textrm{ and }\ (TL)^*=L^*T^*\quad \forall T,L\in\mathscr{L}(E).$$
Moreover it holds that
$$
\langle Tx,\phi\rangle=\langle x,T^*\phi\rangle\quad\forall x\in E, \phi\in E^*\textrm{ and }T\in\mathscr{L}(E).
$$
See, for example, \cite[II.2]{Lor} and \cite[VI.2]{RS}.

The following simple lemma will be useful in our proof of the mean ergodic theorem in $\S\ref{sec2}$.

\begin{Lem}\label{lem1.1}
Let $G$ be an lcH group with a left Haar measure $m_G$ and $\{T_g; g\in G\}$ a family of continuous linear operators of a reflexive Banach space $E$ to itself satisfying conditions:
\begin{itemize}
\item[$(\mathrm{I})$] $T_h(T_gx)=T_{gh}x\ \forall g,h\in G, x\in E$;
\item[$(\mathrm{II})$] for $x\in E$ and $y^*\in E^*$, $g\mapsto\langle T_gx,y^*\rangle$ is a Borel function on $G$;
\item[$(\mathrm{III})$] for $x\in E$ and $y^*\in E^*$, $g\mapsto\langle T_gx,y^*\rangle$ is $m_G$-integrable restricted to any compacta of $G$.
\end{itemize}
Then $\big{\{}T_{g^{-1}}^*, g\in G\big{\}}$ is also such that $(\mathrm{I}), (\mathrm{II})$ and $(\mathrm{III})$ with $E^*$ in place of $E$.
\end{Lem}

\begin{proof}
Let $g,h\in G$ be arbitrary. Then
$$T_{g^{-1}}^*T_{h^{-1}}^*=(T_{h^{-1}}T_{g^{-1}})^*=(T_{g^{-1}h^{-1}})^*=T_{(hg)^{-1}}^*.$$
Thus $T_{g^{-1}}^*$ is contravariant in $g$.
Next for any $y^*\in E^*$ and any $x\in E^{**}=E$, the function
\begin{equation*}
g\mapsto\langle T_{g^{-1}}^*y^*,x\rangle=\langle y^*, T_{g^{-1}}x\rangle
\end{equation*}
is obviously Borel measurable on $G$ by (II). Thus $g\mapsto T_{g^{-1}}^*$ is weakly measurable.
Finally by
$$\left(\int_KT_{g^{-1}}dg\right)^*=\int_KT_{g^{-1}}^*dg$$
for any compact set $K$ of $G$, we can conclude that $\big{\{}T_{g^{-1}}^*\big{\}}_{g\in G}$ satisfies condition (III) on $E^*$.
This completes the proof of Lemma~\ref{lem1.1}.
\end{proof}

\subsection{Vanishing space and range space}
Let $\mathcal{T}=(T_\theta)_{\theta\in\Theta}$ be a family of continuous linear operators of $E$ to itself, where $\Theta$ is not necessarily a directed set. The \textbf{\textit{vanishing space}} of $\mathcal{T}$, written $\mathbf{V}(\mathcal{T})$, is the set of vectors $x\in E$ such that $T_\theta x=0$ for all $\theta\in\Theta$. It is clear that
\begin{gather*}
\mathbf{V}(\mathcal{T})={\bigcap}_{\theta\in\Theta}\ker(T_\theta)
\end{gather*}
is a closed linear subspace of $E$.

By the range of $T\in\mathscr{L}(E)$, write $\mathrm{R}(T)$, is meant the set of vectors $y$ of the form $y=Tx$. Clearly $\mathrm{R}(T)$ is a linear subspace. However, it need not be closed. Let $\mathbf{R}(\mathcal{T})$ represent the minimal closed linear manifold that contains the ranges of all $T_\theta$, i.e.,
\begin{gather*}
\mathbf{R}(\mathcal{T})=\overline{\mathrm{span}\{\mathrm{R}(T_\theta), \theta\in\Theta\}},
\end{gather*}
which is referred to as the \textbf{\textit{closure of the range}} of $\mathcal{T}$.

\subsection{Preliminary theorem}
The following result is very important for proving our Theorem~\ref{thm0.4} in the next section, which is an extension of a theorem of E.~Lorch~\cite[Theorem~1]{Lor39} or \cite[Theorem~8.1]{Lor} for a single operator of the reflexive Banach space $E$.

\begin{Thm}\label{thm1.2}
Let $E$ be a reflexive Banach space and $\mathcal{T}=(T_\theta)_{\theta\in\Theta}$ a family of continuous linear operators of $E$ to itself.
Let $\mathbf{V}(\mathcal{T})$ and $\mathbf{R}(\mathcal{T})$ denote the vanishing space and the closure of the range of $\mathcal{T}$, respectively. Let $\mathcal{T}^*=(T_\theta^*)_{\theta\in\Theta}$, $\mathbf{V}(\mathcal{T}^*)$ and $\mathbf{R}(\mathcal{T}^*)$ be the corresponding structures defined in the dual space $E^*$. Then
\begin{align}
&\mathbf{V}(\mathcal{T})^\perp=\mathbf{R}(\mathcal{T}^*),\quad \mathbf{R}(\mathcal{T})^\perp=\mathbf{V}(\mathcal{T}^*);\label{eq1.1}\\
\intertext{and}
&\mathbf{V}(\mathcal{T}^*)^\perp=\mathbf{R}(\mathcal{T}),\quad \mathbf{R}(\mathcal{T}^*)^\perp=\mathbf{V}(\mathcal{T}).\label{eq1.2}
\end{align}
\end{Thm}

\begin{proof}
First of all we note that $T^{\ast\ast}=T$ for all $T\in\mathscr{L}(E)$ and $E^{**}=E$. We have therefore that $\mathbf{R}(\mathcal{T}^{**})=\mathbf{R}(\mathcal{T})$ and $\mathbf{V}(\mathcal{T}^{**})=\mathbf{V}(\mathcal{T})$ and that (\ref{eq1.1}) implies (\ref{eq1.2}) by applying (\ref{eq1.1}) with $\mathcal{T}^*$ in place of $\mathcal{T}$. Hence to prove the theorem, we need only prove (\ref{eq1.1}).

To start off, let $\phi\in\mathbf{V}(\mathcal{T}^*)$ be arbitrarily given. Then $T_\theta^*(\phi)=0$ for all $\theta\in\Theta$. Now if $x\in E$, then $\langle T_\theta x,\phi\rangle=\langle x,T_\theta^*(\phi)\rangle=0$ for all $\theta\in\Theta$. Hence $\phi$ is orthogonal to the union of the ranges of all $T_\theta$ and (by linearity and continuity) to $\mathbf{R}(\mathcal{T})$. This implies that $\mathbf{R}(\mathcal{T})^\perp\supset\mathbf{V}(\mathcal{T}^*)$.
Now for the opposite inclusion, suppose that $\phi\in\mathbf{R}(\mathcal{T})^\perp$. Then $\langle T_\theta x,\phi\rangle=0$ for any $x\in E$ and any $\theta\in\Theta$. But by $\langle T_\theta x,\phi\rangle=\langle x,T_\theta^*(\phi)\rangle$, we see that $T_\theta^*(\phi)=0$ for all $\theta\in\Theta$ and so $\phi\in\mathbf{V}(\mathcal{T}^*)$. Thus it follows that $\mathbf{R}(\mathcal{T})^\perp\subset\mathbf{V}(\mathcal{T}^*)$ and in consequence we have $\mathbf{R}(\mathcal{T})^\perp=\mathbf{V}(\mathcal{T}^*)$.

Next we let $\phi\in E^*$ and $\theta\in\Theta$ be arbitrarily given and set $\psi=T_\theta^*(\phi)$. If $x\in\mathbf{V}(\mathcal{T})$ is arbitrary, then $\langle x,\psi\rangle=\langle x,T_\theta^*(\phi)\rangle=\langle T_\theta x,\phi\rangle=0$. Thus $\psi\perp\mathbf{V}(\mathcal{T})$. This means that $\mathbf{V}(\mathcal{T})^\perp\supset\mathbf{R}(\mathcal{T}^*)$ since $\mathbf{V}(\mathcal{T})^\perp$ is a closed linear manifold.
To get $\mathbf{V}(\mathcal{T})^\perp\subset\mathbf{R}(\mathcal{T}^*)$, we now assume $\mathbf{V}(\mathcal{T})^\perp\not\subset\mathbf{R}(\mathcal{T}^*)$ and then we can choose a functional $\phi\in\mathbf{V}(\mathcal{T})^\perp$ in $E^*$ satisfying $\phi\not\in\mathbf{R}(\mathcal{T}^*)$. Then there exists a functional in $E^{**}$ which is $0$ restricted to $\mathbf{R}(\mathcal{T}^*)$ and is $1$ on $\phi$ by the Hahn-Banach theorem. Therefore by the reflexivity of $E$, we have the existence of a vector $\x\in E$ such that $\x\perp\mathbf{R}(\mathcal{T}^*)$ and $\phi(\x)=1$. Since $\x\perp\mathbf{R}(\mathcal{T}^*)$, hence $\langle\x,T_\theta^*(\psi)\rangle=0$ for each $\psi\in E^*$ and any $\theta\in\Theta$. Hence $\langle T_\theta\x,\psi\rangle=0$ for any $\psi\in E^*$ and any $\theta\in\Theta$. This means that $\x\in\mathbf{V}(\mathcal{T})$ and since $\phi\in\mathbf{V}(\mathcal{T})^\perp$, we get $\phi(\x)=0$. This contradicts $\phi(\x)=1$. Whence the hypothesis that $\mathbf{V}(\mathcal{T})^\perp\not\subset\mathbf{R}(\mathcal{T}^*)$ is false. Thus $\mathbf{V}(\mathcal{T})^\perp\subset\mathbf{R}(\mathcal{T}^*)$ and hence $\mathbf{V}(\mathcal{T})^\perp=\mathbf{R}(\mathcal{T}^*)$.

The proof of Theorem~\ref{thm1.2} is therefore completed.
\end{proof}

\begin{remark}\label{rem1.3}
The spaces $\mathbf{V}(\mathcal{T}), \mathbf{R}(\mathcal{T}), \mathbf{V}(\mathcal{T}^*)$, and $\mathbf{R}(\mathcal{T}^*)$ all are defined by $\mathcal{T}$ independently of any F{\o}lner net in $G$. This is a crucial point in the proof of Theorem~\ref{thm0.4} in $\S\ref{sec2}$.
\end{remark}

We note that the weak compactness of a bounded set of $E$ plays no a role in the above proof of Theorem~\ref{thm1.2}.

\subsection{Topologies for linear transformations}\label{sec1.4}
To better understand our mean ergodic theorems, it is necessary to know the various topologies on the Banach space $\mathscr{L}(E)$.
\begin{enumerate}
\item The \textbf{\textit{uniform topology}} in $\mathscr{L}(E)$ is the topology defined by the usual operator norm.
\item Let $T,T_1,T_2,\dotsc$ belong to $\mathscr{L}(E)$. Suppose that for every $x\in E$, $\|Tx-T_nx\|_E\to0$ as $n\to\infty$. Then $\{T_n\}$ is said to converge \textbf{\textit{strongly}} to $T$.
\item Let $T,T_1,T_2,\dotsc$ belong to $\mathscr{L}(E)$. Suppose for every vector $x\in E$ and for any functional $\phi\in E^*$, $\langle Tx-T_nx,\phi\rangle\to0$ as $n\to\infty$. Then we say $\{T_n\}$ converges \textbf{\textit{weakly}} to $T$.
\end{enumerate}
Clearly, uniform convergence implies strong convergence; and strong convergence implies weak convergence.
Our mean ergodic theorems all are in the sense of the strong topology on $\mathscr{L}(E)$.
\section{The mean ergodic theorems of continuous operators}\label{sec2}
In this section, we will first prove an abstract mean ergodic theorem (Theorem~\ref{thm2.1} which implies Theorem~\ref{thm0.4} except $p=1$) using Theorem~\ref{thm1.2}. Secondly we shall prove the case $p=1$ of Theorem~\ref{thm0.4} provided $\mu(X)<\infty$.

Before proving, we first point out that our approaches introduced here are only valid for groups, for we need to involve the adjoint operators $T_{g^{-1}}^*$ in our procedure.

\subsection{Abstract mean ergodic theorem}
Let $E$ be a Banach space, and let $G$ be an amenable group with any fixed \textit{left} F{\o}lner net $\{F_\theta, \theta\in\Theta\}$ in it as in $\S\ref{sec0.1}$. We can now generalize Lorch's mean ergodic theorem (\cite[Theorem~9.1]{Lor}) as follows.

\begin{Thm}\label{thm2.1}
Let $\{T_g\colon E\rightarrow E\}_{g\in G}$ be a family of continuous linear operators of a reflexive Banach space $E$ to itself, which is such that:
\begin{enumerate}
\item[$(1)$] $T_gT_h=T_{hg}\ \forall g,h\in G$,
\item[$(2)$] for $x\in E$ and $y^*\in E^*$, $g\mapsto\langle T_gx,y^*\rangle$ is Borel measurable on $G$, and
\item[$(3)$] $\{T_g\}_{g\in G}$ is $m_G$-almost surely uniformly bounded; i.e., $\|T_g\|\le\beta$ for $m_G$-a.e. $g\in G$.
\end{enumerate}
Set
$$\F=\{x\in E\,|\,T_gx=x\ \forall g\in G\}\quad \textrm{and} \quad\N=\overline{\{x-T_gx\,|\,x\in E, g\in G\}}.$$
Then
\begin{gather*}
E=\F\oplus\N \intertext{and}
\frac{1}{|F_\theta|}\int_{F_\theta}T_gxdg\to P(x)\quad \forall x\in E.
\end{gather*}
Here
\begin{gather*}
P\colon \F\oplus\N\rightarrow\F
\end{gather*}
is the projection, $\|P\|\le\beta$; i.e., $A(F_\theta,\centerdot)$ converges strongly to $P(\centerdot)$ in $\mathscr{L}(E)$.
\end{Thm}

This type of mean ergodic theorem was first introduced by J.~von Neumann in 1932 \cite{von} for the case that $\{T_g\colon\mathscr{H}\rightarrow\mathscr{H}\}$ is generated by a single unitary operator $T$ of a Hilbert space $\mathscr{H}$ to itself, and then it was extended to more general spaces than Hilbert space. See, e.g., F.~Riesz in 1938~\cite{Rie} for the case $\|T\|\le1, E=L^p$ where $1<p<\infty$; E.~Lorch in 1939 \cite{Lor39} for the case where $E$ is a reflexive Banach space and $T$ is power bounded; and K.~Yosida and S.~Kakutani for weakly completely continuous $T$ of a Banach space $E$ to itself, respectively, in 1938~\cite{Yos} and \cite{Kak}; also see \cite[Theorem~9.13.1]{Edw}.

We notice here that condition (3) ensures that for any $x\in E$ and any compact subset $K$ of $G$ with $|K|>0$, the \textit{Pettis integral} of the vector-valued function $T_{\centerdot}x\colon g\mapsto T_gx$ over $K$
\begin{gather*}
A(K,x):=\frac{1}{|K|}\int_{K}T_gxdg,
\end{gather*}
is a well-defined vector in $E$ by
\begin{equation*}
\langle A(K,x),y^*\rangle=\frac{1}{|K|}\int_K\langle T_gx,y^*\rangle dg\quad \forall y^*\in E^*.
\end{equation*}
See, e.g., \cite[Proposition~8.14.11]{Edw} and \cite[Chap.~III]{HP}.  Thus $A(K,\centerdot)\colon E\rightarrow E$ makes sense for any compacta $K$ of $G$, which is linear with respect to $x\in E$ with
$\|A(K,\centerdot)\|\le\sup_{g\in K}\|T_g\|$.

\begin{proof}[Proof of Theorem~\ref{thm2.1}]
Let $\mathcal{T}=\{I-T_g, g\in G\}$ and then $\F=\mathbf{V}(\mathcal{T})$ and $\N=\mathbf{R}(\mathcal{T})$ which are independent of $\{F_\theta; \theta\in\Theta\}$, where $I$ is the identity operator of $E$. Suppose first that $x\in E$ is in $\F$; thus $T_gx=x$ for all $g\in G$ and then
$$\frac{1}{|F_\theta|}\int_{F_\theta}T_gxdg=x\quad \forall \theta\in\Theta.$$
Thus restricted to $\F$, we have
$$
\lim_{\theta\in\Theta}\frac{1}{|F_\theta|}\int_{F_\theta}T_gxdg=x.
$$
Now let $y\in\N$ be arbitrary. This means that for any $\epsilon>0$, there are finite number of elements, say $g_1,\dotsc,g_\ell\in G$, and vectors $z_1,\dotsc,z_\ell, z\in E$ such that
$$
y=(z_1-T_{g_1}z_1)+\dotsm+(z_\ell-T_{g_\ell}z_\ell)+z,\quad\|z\|_E<\epsilon.
$$
By the asymptotical invariance of the F{\o}lner net $\{F_\theta; \theta\in\Theta\}$ and the a.s. uniform boundedness of $\{T_g\}$, an easy calculation shows that
\begin{gather*}
\limsup_{\theta\in\Theta}\left\|\frac{1}{|F_\theta|}\int_{F_\theta}T_gydg\right\|_E\le\beta\epsilon.
\end{gather*}
It is now clear that $y\in\N$ implies that
$$\frac{1}{|F_\theta|}\int_{F_\theta}T_gydg\to0$$
since $\epsilon$ is arbitrary.

Meanwhile the above argument shows that $\F\cap\N=\{0\}$ and so $\F+\N=\F\oplus\N$. Suppose that $x\in\F$ and $y\in\N$. Then
\begin{gather*}
\frac{1}{|F_\theta|}\int_{F_\theta}T_g(x+y)dg\to x\quad \textrm{and} \quad \|x\|\le\beta\|x+y\|.
\end{gather*}
We have therefore shown that
$A(F_\theta,\centerdot)$, restricted to vectors in $\F\oplus\N$, converges to the projection $P$ of $\F\oplus\N$ onto $\F$ for which $\|P\|\le\beta$.

To prove Theorem~\ref{thm2.1}, it remains to show that $E=\F\oplus\N$. By similar arguments, according to Lemma~\ref{lem1.1} we see that for the adjoint operator group
\begin{gather*}
\left\{T_{g^{-1}}^*\colon E^*\rightarrow E^*\right\}_{g\in G},
\end{gather*}
which is such that
\begin{gather*}
T_{g^{-1}}^*T_{h^{-1}}^*=T_{(hg)^{-1}}^*\quad\textrm{and}\quad (T_g^*)^{-1}=T_{g^{-1}}^*\quad \forall g,h\in G,
\end{gather*}
if we restrict to the manifold $\F(\mathcal{T}^*)\oplus\N(\mathcal{T}^*)$ that is associated to
\begin{gather*}
\mathcal{T}^*=\left\{I^*-T_{g^{-1}}^*\right\}_{g\in G}=\left\{I^*-{T_{g}^*}^{-1}\right\}_{g\in G},
\end{gather*}
there follows that $\frac{1}{|F_\theta|}\int_{F_\theta}T_{g^{-1}}^*(\centerdot)dg$ converges to the projection of $\F(\mathcal{T}^*)\oplus\N(\mathcal{T}^*)$ onto $\F(\mathcal{T}^*)$.
In particular, $\F(\mathcal{T}^*)\cap\N(\mathcal{T}^*)=\{0\}$.

Now to the contrary suppose that there is some $x\not=0$ in $E$ such that $x\not\in\F(\mathcal{T})\oplus\N(\mathcal{T})$. Then there exists some functional $\phi\in E^*$ such that $\phi\perp\F\oplus\N$ and $\phi(x)=1$ by the Hahn-Banach theorem; hence $\phi\not=0$. But $\phi\in\F^\perp=\N(\mathcal{T}^*)$ and $\phi\in\N^\perp=\F(\mathcal{T}^*)$ by virtue of Theorem~\ref{thm1.2} and Remark~\ref{rem1.3}. Hence $\phi=0$. This contradiction shows that $E=\F\oplus\N$.

The proof of Theorem~\ref{thm2.1} is thus completed.
\end{proof}

We note here that in Theorem~\ref{thm2.1} the projection $P$ is independent of the F{\o}lner net $\{F_\theta; \theta\in\Theta\}$ in $G$. This thus completes the proof of our Theorem~\ref{thm0.4} in the case $1<p<\infty$ since $L^p(X,\mathscr{X},\mu)$, for $1<p<\infty$, are reflexive Banach spaces.

\subsection{Proof of Theorem~\ref{thm0.4}}\label{sec2.2}
We need only say some words for the case $p=1$ under the condition $\mu(X)<\infty$. In fact, let
\begin{gather*}
\F_\mu^1=\{\phi\in L^1(X,\mathscr{X},\mu)\,|\,\phi=T_g\phi\ \forall g\in G\}\quad \textrm{and}\quad\N_\mu^1=\overline{\{\psi-T_g\psi\,|\,\psi\in L^1(X,\mathscr{X},\mu), g\in G\}}.
\end{gather*}
Then we can easily see that $\F_\mu^1\cap\N_\mu^1=\{0\}$ and that for any $\phi=\phi_1+\phi_2\in\F_\mu^1\oplus\N_\mu^1$ with $\phi_1\in\F_\mu^1$ and $\phi_2\in\N_\mu^1$ there follows $A(F_\theta,\phi)\to\phi_1$ with $\|\phi_1\|_1\le\|\phi\|_1$ and $A(F_\theta,\phi_2)\to0$ in $L^1$-norm.
Now let $\phi\in L^1(X,\mathscr{X},\mu)$ be arbitrary; and then $\exists\psi_n\in L^2(X,\mathscr{X},\mu)$ with $\psi_n\to\phi$ in $L^1$-norm. Since $\psi_n=P(\psi_n)+(\psi_n-P(\psi_n))$ where $\psi_n-P(\psi_n)\in\N_\mu^1$ and
$\|P(\psi_m)-P(\psi_n)\|_1\le\|\psi_m-\psi_n\|_1$ for all $m,n\ge1$,
it follows that $\varphi:=L^1\textrm{-}\lim_{n\to\infty}P(\psi_n)\in\F_\mu^1$ and $\phi-\varphi\in\N_\mu^1$. Thus,
$$L^1(X,\mathscr{X},\mu)=\F_\mu^1\oplus\N_\mu^1.$$
This hence completes the proof of the case $p=1$ of Theorem~\ref{thm0.4} whenever $\mu(X)<\infty$.

\subsection{A mean ergodic theorem for product amenable groups}

Let $G_1,\dotsc,G_l$ be $l$ $\sigma$-compact amenable groups which respectively have fixed left Haar measures $m_1,\dotsc,m_l$, where $l\ge2$ is an integer. By $m$ we denote the usual $l$-fold product measure $m_1\otimes\dotsm\otimes m_l$ on the $l$-fold product space $G_1\times\dotsm\times G_l$.
Let $\{T_{g_i}\}_{g_i\in G_i}$, $1\le i\le l$, be families of continuous linear operators of a same reflexive Banach space $E$ to itself satisfying conditions (1), (2) and (3) of Theorem~\ref{thm2.1}.

Then by Theorem~\ref{thm2.1}, for any F{\o}lner sequences $\left(F_n^{(i)}\right)_{n=1}^\infty$ in $G_i, 1\le i\le l$, we can obtain in the sense of the strong operator topology that
\begin{equation*}
P_i(\centerdot)=\lim_{n\to\infty}\frac{1}{m_i\left(F_n^{(i)}\right)}\int_{F_n^{(i)}}T_{g_i}(\centerdot)dg_i.
\end{equation*}
Under the above situation, the following result generalizes \cite{NW}.

\begin{Thm}\label{thm2.2}
There follows that
\begin{equation*}
P_1\dotsm P_l=\lim_{n\to\infty}\frac{1}{m\left(F_n^{(1)}\times\dotsm\times F_n^{(l)}\right)}\int_{F_n^{(1)}\times\dotsm\times F_n^{(l)}}T_{g_1}\dotsm T_{g_l}dg_1\dotsm dg_l.
\end{equation*}
\end{Thm}

\begin{proof}
First of all, we note that for any bounded linear operator $L\colon E\rightarrow E$,
\begin{gather*}
\int_{F_n^{(i)}}LT_{g_i}dg_i=L\int_{F_n^{(i)}}T_{g_i}dg_i\intertext{and}\int_{F_n^{(i)}}T_{g_i}Ldg_i=\left(\int_{F_n^{(i)}}T_{g_i}dg_i\right)L
\end{gather*}
for all $1\le i\le l$; and $m\left(F_n^{(1)}\times\dotsm\times F_n^{(l)}\right)=m_1\left(F_n^{(1)}\right)\dotsm m_l\left(F_n^{(1)}\right)$.
Let $l=2$. Then by Fubini's theorem,
\begin{equation*}\begin{split}
\lefteqn{\lim_{n\to\infty}\frac{1}{m\big{(}F_n^{(1)}\times F_n^{(2)}\big{)}}\int_{F_n^{(1)}\times F_n^{(2)}}T_{g_1}T_{g_2}dg_1dg_2}\\
&=\lim_{n\to\infty}\frac{1}{m_1\big{(}F_n^{(1)}\big{)}}\int_{F_n^{(1)}}\left[\frac{1}{m_2\big{(}F_n^{(2)}\big{)}}\int_{F_n^{(2)}}T_{g_1}T_{g_2}dg_2\right]dg_1\\
&=\lim_{n\to\infty}\frac{1}{m_1\big{(}F_n^{(1)}\big{)}}\int_{F_n^{(1)}}T_{g_1}\left[\frac{1}{m_2\big{(}F_n^{(2)}\big{)}}\int_{F_n^{(2)}}T_{g_2}dg_2\right]dg_1\\
&=\lim_{n\to\infty}\left[\frac{1}{m_1\big{(}F_n^{(1)}\big{)}}\int_{F_n^{(1)}}T_{g_1}dg_1\right]\left[\frac{1}{m_2\big{(}F_n^{(2)}\big{)}}\int_{F_n^{(2)}}T_{g_2}dg_2\right]\\
&=P_1P_2.
\end{split}\end{equation*}
By induction on $l$, we can easily prove the statement for any $2\le l<\infty$. This completes the proof of Theorem~\ref{thm2.2}.
\end{proof}

Whenever $G_1=\dotsm=G_l=\mathbb{Z}$, then this result reduces to an analogous case of \cite[Theorem~1]{NW}.

For a single continuous linear operator $T$ of $E$, Cox~\cite{Co} showed that if $\sup_n\|I-T^n\|<1$, then $T=I$. Similar to \cite[Theorem~4]{NW}, using Theorem~\ref{thm2.2} we can generalize and strengthen Cox's theorem as follows.

\begin{cor}\label{cor2.3}
Under the same situation of Theorem~\ref{thm2.2}, let
$\|I-P_1\dotsm P_l\|<1$;
then $T_g=I$ for all $g\in G_i, 1\le i\le l$.
\end{cor}

\begin{proof}
Let $\|I-P_1\dotsm P_l\|=\alpha<1$. Then $P_1\dotsm P_l$ is invertible and $\|P_1\dotsm P_l\|\le(1-\alpha)^{-1}$. Since for any $h\in G_l$ and any vector $x\in E$ we have \begin{equation*}
\begin{split}
\|x-T_h(x)\|&=\|(I-T_h)(x)\|=\|(P_1\dotsm P_l)^{-1}(P_1\dotsm P_l)(I-T_h)(x)\|\\
&\le\frac{1}{1-\alpha}\|P_1\dotsm P_l(I-T_h)(x)\|,
\end{split}
\end{equation*}
the last term is equal to $0$ by Theorem~\ref{thm2.2}. So $T_h=I$ for each $h\in G_l$. Similarly, it follows $T_h=I$ for all $h\in G_i$ for $1\le i\le l-1$. This proves the corollary.
\end{proof}

\section{Pointwise ergodic theorem for amenable groups}\label{sec3}
This section will be devoted to proving our $L^\infty$-pointwise ergodic theorems (Theorems~\ref{thm0.5} and \ref{thm0.6} stated in $\S\ref{sec0.2}$) by using classical analysis and our $L^p$-mean ergodic theorem (Theorem~\ref{thm0.4}).
\subsection{Arzel\'{a}-Ascoli theorem}
Recall that a topological space $X$ is call a \textbf{\textit{k}-\textit{space}}, provided that if a subset $A$ of $X$ intersects each closed compact set in a closed set then $A$ is closed. The two most important examples of \textit{k}-spaces are given in the following.

\begin{Lem}[{\cite[Theorem~7.13]{Kel}}]\label{lem3.1}
If $X$ is a Hausdorff space which is either locally compact or satisfies the first axiom of countability, then $X$ is a \textit{k}-space.
\end{Lem}
Thus each metric space is a \textit{k}-space like $L^\infty(X,\mathscr{X})$ with the uniform norm $\|\cdot\|_\infty$ over any measurable space $(X,\mathscr{X})$.

A \textbf{\textit{uniformity}} for a set $X$ is a non-void family $\mathscr{U}$ of subsets of $X\times X$ such that: (a) each member of $\mathscr{U}$ contains the diagonal $\Delta(X)$; (b) if $U\in\mathscr{U}$, then $U^{-1}=\{(y,x)\,|\,(x,y)\in U\}\in\mathscr{U}$; (c) if $U\in\mathscr{U}$, then $V\circ V\subseteq U$ for some $V$ in $\mathscr{U}$; (d) if $U$ and $V$ are members of $\mathscr{U}$, then $U\cap V\in\mathscr{U}$; and (e) if $U\in\mathscr{U}$ and $U\subset V\subset X\times X$, then $V\in\mathscr{U}$. The pair $(X,\mathscr{U})$ is called a \textbf{\textit{uniform space}}.

If $(X,\mathscr{U})$ is a uniform space, then the \textbf{\textit{uniform topology}} $\mathfrak{T}$ of $X$ associated to $\mathscr{U}$ is the family of all subsets $T$ of $X$ such that for each $x\in T$ there is $U\in\mathscr{U}$ with $U[x]\subset T$. In this case, $X$ is called a \textbf{\textit{uniform topological space}}.

\begin{Lem}[{\cite[Theorem~6.10]{Kel}}]\label{lem3.2}
Let $X_\alpha, \alpha\in A$, be a family of uniform topological space; then $\prod_{\alpha\in A}X_\alpha$, with the product topology (in other words, the pointwise topology), is a uniform topological space.
\end{Lem}

Thus, the product space $\mathbb{R}^X$ with the topology $\mathfrak{T}_{\textrm{p-c}}$ of pointwise convergence is a uniform topological space.

Now it is time to formulate our important tool---the classical Arzel\'{a}-Ascoli theorem (cf., e.g.,~\cite[Theorem~7.18]{Kel}).
\begin{AA-I}
Let $C(X,Y)$ be the family of all continuous functions on a \textit{k}-space $X$ which is either Hausdorff or regular to a Hausdorff uniform topological space $Y$, and let $C(X,Y)$ be equipped with the compact-open topology. Then a subfamily $F$ of $C(X,Y)$ is compact if and only if
\begin{enumerate}
\item[$(a)$] $F$ is closed in $C(X,Y)$,
\item[$(b)$] $F$ is equicontinuous on every compact subset of $X$, and
\item[$(c)$] $F[x]=\{f(x)\,|\, f\in F\}$ has compact closure in $Y$ for each $x\in X$.
\end{enumerate}
\end{AA-I}


\subsection{$L^\infty$-Pointwise convergence everywhere}\label{sec3.2}
Let $(X,\mathscr{X})$ be a measurable space and $G$ an amenable group with the identity $e$. We now consider a Borel dynamical system, i.e., $G$ acts Borel measurably on $X$,
$$
T\colon G\times X\rightarrow X\quad\textrm{ or write }\quad G\curvearrowright_TX.
$$
That is to say, the action map $T$ satisfies the following conditions:
\begin{enumerate}
\item[(1)] $T\colon (g,x)\mapsto T(g,x)$ is jointly measurable; and
\item[(2)] $T_ex=x$ and $T_{g_1}T_{g_2}x=T_{g_1g_2}x$ for all $x\in X$ and $g_1,g_2\in G$.
\end{enumerate}
In such case, $(X,\mathscr{X})$ is called a \textit{\textbf{Borel $G$-space}}. If $X$ is a topological space with $\mathscr{X}=\mathscr{B}_X$ and $T(g,x)$ is jointly continuous, then $X$ is called a \textit{\textbf{topological $G$-space}}.

By $\mathscr{K}_G$ we denote the family of all compacta of $G$. For any $K\in\mathscr{K}_G$, define the continuous linear operator
\begin{gather*}
A(K,\centerdot)\colon L^\infty(X,\mathscr{X})\rightarrow L^\infty(X,\mathscr{X}),
\end{gather*}
which is called the \textit{\textbf{ergodic average} over $K$}, by
\begin{gather*}
A(K,\varphi)(x)=\frac{1}{|K|}\int_KT_g\varphi(x)dg\quad \forall \varphi\in L^\infty(X,\mathscr{X}).
\end{gather*}
Clearly, the operator norms $\|A(K,\centerdot)\|\le 1$ and $\|T_g\|=1$.

Next we will simply write $(E,\|\cdot\|_E)=(L^\infty(X,\mathscr{X}),\|\cdot\|_\infty)$, which is a Banach space, and let $Y=\mathbb{R}^X$ endowed with the topology of pointwise convergence (i.e. the product topology). Then $E$ is a \textit{k}-space by Lemma~\ref{lem3.1} and $Y$ is a Hausdorff uniform topological space by Lemma~\ref{lem3.2}.
In fact, $Y$ is a vector space by $\gamma(r_x)_{x\in X}+(r_x^\prime)_{x\in X}=(\gamma r_x+r_x^\prime)_{x\in X}$ for $\gamma\in\mathbb{R}$ and $r,r^\prime\in Y$.

Now, for any $K\in\mathscr{K}_G$, we can define the function
$\mathcal{A}(K,\centerdot)\colon E\rightarrow Y$
by the following way:
\begin{gather}
E\ni\varphi\mapsto\mathcal{A}(K,\varphi)=\left(A(K,\varphi)(x)\right)_{x\in X}\in Y.
\end{gather}
Since for any $\varphi,\psi\in E$ and $x\in X$ there follows
\begin{equation}\label{eq3.2}
|A(K,\varphi)(x)-A(K,\psi)(x)|\le\|\varphi-\psi\|_E,
\end{equation}
we can see that $\mathcal{A}(K,\centerdot)$ is continuous on $E$ valued in $Y$ and
$\{\mathcal{A}(K,\centerdot), K\in\mathscr{K}_G\}$ is an equicontinuous subfamily of $C(X,Y)$, noting that here $C(X,Y)$ is equipped with the compact-open topology.

Under these facts we can obtain the following $L^\infty$-pointwise convergence theorem, which is completely independent of ergodic theory of $G\curvearrowright_TX$.

\begin{Thm}\label{thm3.3}
Given any F{\o}lner net $\{F_{\theta^\prime}^\prime; \theta^\prime\in\Theta^\prime\}$ in $G$ and any Borel $G$-space $X$, one can find a F{\o}lner subnet $\{F_\theta; \theta\in\Theta\}$ of $\{F_{\theta^\prime}^\prime; \theta^\prime\in\Theta^\prime\}$ in $G$ and a continuous linear function $\mathcal{A}(\centerdot)\colon E\rightarrow Y$ such that in the sense of Moore-Smith limit
\begin{gather*}
\lim_{\theta\in\Theta}\mathcal{A}(F_{\theta},\centerdot)=\mathcal{A}(\centerdot)\quad \textrm{in }C(E,Y)
\intertext{with the property: for each $\varphi\in E$,}
\lim_{\theta\in\Theta}A(F_{\theta},\varphi)(x)=\mathcal{A}(\varphi)(x)\ \forall x\in X\quad \textrm{and}\quad \mathcal{A}(\varphi)=\mathcal{A}(T_g\varphi)\ \forall g\in G.
\end{gather*}
Moreover, if $\varphi_n\to\psi$ as $n\to\infty$ in $(E,\|\cdot\|_E)$, then $\mathcal{A}(\varphi_n)(x)\to\mathcal{A}(\psi)(x)$ in $\mathbb{R}$ for all $x\in X$.
\end{Thm}

\begin{proof}
Let $F=\mathrm{Cl}_{\textrm{cpt-op}}(\{\mathcal{A}(F_{\theta^\prime}^\prime,\centerdot);\theta^\prime\in\Theta^\prime\})$ be the closure of the family $\{\mathcal{A}(F_{\theta^\prime}^\prime,\centerdot);\theta^\prime\in\Theta^\prime\}$ in $C(E,Y)$ under the compact-open topology.
To employ the Arzel\'{a}-Ascoli Theorem, in light of Lemmas~\ref{lem3.1} and \ref{lem3.2}, it is sufficient to check conditions $(b)$ and $(c)$.

For that, let $L\colon E\rightarrow Y$ be any cluster point of $\{\mathcal{A}(F_{\theta^\prime}^\prime,\centerdot);\theta^\prime\in\Theta^\prime\}$ in $C(E,Y)$, $\varphi\in E, x\in X$ and $\varepsilon>0$.  Define an open neighborhood of $L(\varphi)$ in $Y$
$$
V_\varphi:=V\left(L(\varphi);x,\frac{\varepsilon}{3}\right)=\left\{\xi\in\mathbb{R}^X\,\big{|}\,|\xi(x)-L(\varphi)(x)|<\frac{\varepsilon}{3}\right\},
$$
and for any $\psi\in E$ with $\|\varphi-\psi\|_E<\frac{\varepsilon}{3}$ define another open neighborhood of $L(\psi)$ in $Y$
$$
V_\psi:=V\left(L(\psi);x,\frac{\varepsilon}{3}\right)=\left\{\xi\in\mathbb{R}^X\,\big{|}\,|\xi(x)-L(\psi)(x)|<\frac{\varepsilon}{3}\right\}.
$$
Since
$$
\mathscr{U}(L;\varphi,V_\varphi)=\left\{S\in C(E,Y)\,|\, S(\varphi)\in V_\varphi\right\}\quad \textrm{and}\quad  \mathscr{U}(L;\psi,V_\psi)=\left\{S\in C(E,Y)\,|\, S(\psi)\in V_\psi\right\}
$$
both are open neighborhoods of the point $L$ in $C(E,Y)$, then
\begin{gather*}
\mathscr{V}_L:=\left(\mathscr{U}(L;\varphi,V_\varphi)\cap\mathscr{U}(L;\psi,V_\psi)\right)\cap\{\mathcal{A}(F_{\theta^\prime}^\prime,\centerdot),\theta^\prime\in\Theta^\prime\}\not=\varnothing.
\end{gather*}
Now take any $S$ from the above non-void intersection $\mathscr{V}_L$. So
\begin{align*}
|L(\varphi)(x)-L(\psi)(x)|&\le|L(\varphi)(x)-S(\varphi)(x)|+|S(\varphi)(x)-S(\psi)(x)|+|S(\psi)(x)-L(\psi)(x)|\\
&<\varepsilon,
\end{align*}
which implies that $F$ is equicontinuous; that is to say, condition $(b)$ holds.

In addition, to check condition $(c)$, for any $\varphi\in E$ with $\|\varphi\|_E\le N-\varepsilon$ and any $x\in X$, we have
$$
|L(\varphi)(x)|\le|L(\varphi)(x)-S(\varphi)(x)|+|S(\varphi)(x)|\le\frac{\varepsilon}{3}+\|\varphi\|_E.
$$
Therefore
$$
F[\varphi]:=\left\{\mathcal{A}(F_{\theta^\prime}^\prime,\varphi)\,|\,\theta^\prime\in\Theta^\prime\right\}\subseteq[-N,N]^X.
$$
Because $[-N,N]^X$ is a closed compact subset of $Y$, $F[\varphi]$ has a compact closure in $Y$ and so condition $(c)$ holds.

If the Moore-Smith limit $\mathcal{A}(\varphi)$ exists, then it is easy to see $\mathcal{A}(\varphi)=\mathcal{A}(T_g\varphi)$ for $g\in G$ by the asymptotical invariance of F{\o}lner nets. Therefore, the statement follows immediately from the Arzel\'{a}-Ascoli Theorem.
\end{proof}

It should be noted that although the Moore-Smith limit $\mathcal{A}(\varphi)$ is a function on $X$ bounded by $\|\varphi\|_\infty$, yet it does not need to be $\mathscr{X}$-measurable in general for $Y=\mathbb{R}^X$ is only with the pointwise topology.
\subsection{Pointwise ergodic theorem}\label{sec3.3}
Now we are ready to prove Theorem~\ref{thm0.5} and Theorem~\ref{thm0.6}.
\subsubsection{Proof of Theorem~\ref{thm0.5}}
Let $G$ be any given amenable group and $\{F_{\theta^\prime}^\prime;\theta^\prime\in\Theta^\prime\}$ a F{\o}lner net in $G$. Theorem~\ref{thm3.3} implies that for any Borel action $G\curvearrowright_T(X,\mathscr{X})$, one always can find some F{\o}lner subnet $\{F_{\theta};\theta\in\Theta\}$ from the given net $\{F_{\theta^\prime}^\prime;\theta^\prime\in\Theta^\prime\}$, which is $L^\infty$-admissible for $G\curvearrowright_T(X,\mathscr{X})$; namely,
\begin{gather*}
A_T(F_\theta,\varphi)(x):=\frac{1}{|F_\theta|}\int_{F_\theta}\varphi(T_gx)dg\xrightarrow[]{\textit{Moore-Smith}}\varphi^*(x)\quad \forall x\in X
\end{gather*}
for any $\varphi\in L^\infty(X,\mathscr{X})$. In view of this, we can prove Theorem~\ref{thm0.5} as follows.

\begin{proof}
In contrast to Theorem~\ref{thm0.5}, for \textbf{\textit{arbitrary}} F{\o}lner subnet $\{F_{\theta};\theta\in\Theta\}$ of $\{F_{\theta^\prime}^\prime;\theta^\prime\in\Theta^\prime\}$ in $G$, there are two disjoint families of Borel $G$-actions
\begin{align*}
&F_1:=\left\{G\curvearrowright_{T_\gamma}(X_\gamma,\mathscr{X}_\gamma);\gamma\in\Gamma\right\}\\
\intertext{and}
&F_2:=\left\{G\curvearrowright_{T_\lambda}(X_\lambda,\mathscr{X}_\lambda);\lambda\in\Lambda\right\},
\end{align*}
where $\{F_{\theta};\theta\in\Theta\}$ is $L^\infty$-admissible for $G\curvearrowright_{T_\gamma}(X_\gamma,\mathscr{X}_\gamma)$ for each $\gamma\in\Gamma$ and
$\{F_{\theta};\theta\in\Theta\}$ is not $L^\infty$-admissible for $G\curvearrowright_{T_\lambda}(X_\lambda,\mathscr{X}_\lambda)$ for each $\lambda\in\Lambda$,
such that every Borel $G$-action $G\curvearrowright_T(X,\mathscr{X})$ belongs either to $F_1$ or to $F_2$ and with $F_2\not=\varnothing$ (and in fact $F_1\not=\varnothing$).

In order to arrive at a contradiction, noting that $F_1\not=\varnothing$ and so $F_2$ is not the class of all Borel $G$-actions, we can set
\begin{gather*}
\X={\prod}_{\lambda\in\Lambda}X_\lambda,\quad \bX={\bigotimes}_{\lambda\in\Lambda}\mathscr{X}_\lambda
\end{gather*}
and then define the Borel action of $G$ on $\X$ as follows:
\begin{gather*}
\T\colon G\times\X\rightarrow\X\ \textrm{ or }\ G\curvearrowright_{\T}\X;\quad (g,(x_\lambda)_{\lambda\in\Lambda})\mapsto \T_g\big{(}(x_\lambda)_{\lambda\in\Lambda}\big{)}=\big{(}T_{\lambda,g}x_\lambda\big{)}_{\lambda\in\Lambda}.
\end{gather*}
Now every $\varphi\in L^\infty(X_{\lambda_0},\mathscr{X}_{\lambda_0})$, for any $\lambda_0\in\Lambda$, may be thought of as an element $\bvarphi$ in $L^\infty(\X,\bX)$ by the following way:
\begin{gather*}
\bvarphi(\x)=\varphi(x_{\lambda_0})\quad \forall \x=(x_\lambda)_{\lambda\in\Lambda}\in\X.
\end{gather*}
Clearly,
$$
A_{\T}(K,\bvarphi)(\x)=A_{T_{\lambda_0}}(K,\varphi)(x_{\lambda_0}),\quad \forall \x=(x_\lambda)_{\lambda\in\Lambda}\in\X,
$$
over any compacta $K$ of $G$.

Then by Theorem~\ref{thm3.3} with $G\curvearrowright_{\T}(\X,\bX)$ in place of $G\curvearrowright_T(X,\mathscr{X})$, we can further select out a F{\o}lner subnet $\{F_{\theta^{\prime\prime}}^{\prime\prime};\theta^{\prime\prime}\in\Theta^{\prime\prime}\}$ from the F{\o}lner subnet $\{F_\theta; \theta\in\Theta\}$ of $\{F_{\theta^\prime}^\prime;\theta^\prime\in\Theta^\prime\}$ in $G$, which is \textit{$L^\infty$-admissible} for $G\curvearrowright_{\T}\X$. Clearly, $\{F_{\theta^{\prime\prime}}^{\prime\prime};\theta^{\prime\prime}\in\Theta^{\prime\prime}\}$ in $G$ is $L^\infty$-admissible for $G\curvearrowright_{T_\lambda}X_\lambda$ for each $\lambda\in\Lambda$.

Let $\lambda_0\in\Lambda$ be arbitrarily given. Let $A(\centerdot)\colon L^\infty(X_{\lambda_0},\mathscr{X}_{\lambda_0})\rightarrow\mathbb{R}^{X_{\lambda_0}}$ be given by $A(\varphi)=\mathcal{A}(\bvarphi)$, where $\mathcal{A}(\centerdot)$ is defined by Theorem~\ref{thm3.3} for $G\curvearrowright_{\T}\X$ over $\{F_{\theta^{\prime\prime}}^{\prime\prime};\theta^{\prime\prime}\in\Theta^{\prime\prime}\}$. Then $A(\varphi)$ is linear continuous in the sense:
\begin{gather*}
A(a\varphi+\psi)=aA(\varphi)+A(\psi)\quad \forall a\in\mathbb{R}\textrm{ and }\varphi,\psi\in L^\infty(X_{\lambda_0},\mathscr{X}_{\lambda_0})
\intertext{and}A(\varphi_n)(x)\rightarrow A(\psi)(x)\ \forall x\in X_{\lambda_0}\quad \textrm{as }\varphi_n\to\psi\textrm{ in }(L^\infty(X_{\lambda_0},\mathscr{X}_{\lambda_0}),\|\cdot\|_\infty).
\end{gather*}
Since $\{F_{\theta^{\prime\prime}}^{\prime\prime};\theta^{\prime\prime}\in\Theta^{\prime\prime}\}$ is itself a left F{\o}lner net in $G$, hence $A(\varphi)=A(T_g\varphi)$ for any $g$ in $G$ and all $\varphi$ in $L^\infty(X_{\lambda_0},\mathscr{X}_{\lambda_0})$.

Finally, by the same manner as defining $F_1$ and $F_2$ above, we now define two new families of Borel $G$-actions $F_1^{\prime\prime}$ and $F_2^{\prime\prime}$ over $\{F_{\theta^{\prime\prime}}^{\prime\prime};\theta^{\prime\prime}\in\Theta^{\prime\prime}\}$ instead of $\{F_\theta; \theta\in\Theta\}$. Since $\{F_{\theta^{\prime\prime}}^{\prime\prime};\theta^{\prime\prime}\in\Theta^{\prime\prime}\}$ is a subnet of $\{F_\theta; \theta\in\Theta\}$, we have $F_1\subseteq F_1^{\prime\prime}$. But $F_2\subseteq F_1^{\prime\prime}$ and $F_1\cup F_2$ contains all Borel $G$-actions, this implies that $F_1^{\prime\prime}=\varnothing$, a contradiction.

This proves Theorem~\ref{thm0.5}.
\end{proof}

\subsubsection{Proof of Theorem~\ref{thm0.6}}
Let $G$ be a $\sigma$-compact amenable group. We now assume $\mathcal{M}(G\curvearrowright_TX)\not=\varnothing$ where $X$ is a Borel $G$-space with a Borel structure, i.e., $\sigma$-field $\mathscr{X}$ on $X$.
Then $L^\infty(X,\mathscr{X})\subset L^p(X,\mathscr{X},\mu), 1\le p\le\infty$, for any $\mu\in\mathcal{M}(G\curvearrowright_TX)$.

The following statement is clearly contained in the standard proof of Completeness of $L^p$, which also holds in the case $\mu(X)=\infty$.

\begin{Lem}\label{lem3.4}
If $1\le p\le\infty$ and if $\varphi_n\to\varphi$ in $L^p(X,\mathscr{X},\mu)$ as $n\to\infty$, then $\{\varphi_n\}_1^\infty$ has a subsequence which converges pointwise $\mu$-almost everywhere to $\varphi$.
\end{Lem}

There is no a net version of Lemma~\ref{lem3.4} in the literature yet. Now we are ready to complete the proof of Theorem~\ref{thm0.6} combining Theorem~\ref{thm0.4} and Theorem~\ref{thm0.5}.

\begin{proof}[Proof of Theorem~\ref{thm0.6}]
Let $A\colon\varphi\mapsto\varphi^*$ from $L^\infty(X,\mathscr{X})$ to $\mathbb{R}^X$ over some $L^\infty$-admissible F{\o}lner net $\{F_\theta;\theta\in\Theta\}$ for $G$ by Theorem~\ref{thm0.5}.
Next let $\mu\in\mathcal{M}(G\curvearrowright_TX)$ and $\varphi\in L^\infty(X,\mathscr{X})$ be arbitrarily given. Then from Theorem~\ref{thm0.4}, it follows that under the $L^1$-norm,
\begin{align*}
\lim_{\theta\in\Theta}A(F_\theta,\varphi)&=P(\varphi)\in\F_\mu^1=\left\{\phi\in L^1(X,\mathscr{X},\mu)\,|\,T_g\phi=\phi\ \forall g\in G\right\};\\
\bigg{(}&=\int_X\varphi d\mu\quad \textit{if }\mu\textit{ ergodic}\bigg{)}.
\end{align*}
Since $(L^1(X,\mathscr{X},\mu),\|\centerdot\|_1)$ is a Banach space, we can choose a sequence $\{F_{\theta_n}; n=1,2,\dotsc\}$ from $\{F_\theta;\theta\in\Theta\}$ such that
$$
\lim_{n\to\infty}A(F_{\theta_n},\varphi)=P(\varphi).
$$
So by Lemma~\ref{lem3.4}, we can obtain that
$\varphi^*(x)=P(\varphi)(x)$ for $\mu$-\textit{a.e.}~$x\in X$.

The proof of Theorem~\ref{thm0.6} is thus completed.
\end{proof}

\begin{remark}
Let $G$ be a $\sigma$-compact amenable group with an $L^\infty$-admissible F{\o}lner sequence $\{F_n; n=1,2,\dotsc\}$. We note that in Theorem~\ref{thm0.6}, for any $\varphi\in L^\infty(X,\mathscr{X},\mu)$, we may take $\varphi^\prime$ a version of $\varphi$ with $\varphi^\prime\in L^\infty(X,\mathscr{X})$ such that
\begin{gather*}
\|\varphi^\prime\|_\infty=\|\varphi\|_{\infty,\mu}\quad \textrm{and}\quad \psi:=|\varphi^\prime-\varphi|=0\textit{ a.e.}~(\mu).
\end{gather*}
Then $T_g\psi=\psi$ \textit{a.e.} for each $g\in G$. By Fubini's theorem, we can take some $\mu$-conull $X_0\in\mathscr{X}$ such that for any $x_0\in X_0$ and for $m_G$-\textit{a.e.} $g\in G$, $\psi(T_gx_0)=\psi(x_0)\ (=0)$.
Thus
\begin{gather*}
\lim_{n\to\infty}A(F_n,\psi)=\psi=0\textit{ a.e.}
\end{gather*}
This means that
\begin{gather*}
{\varphi^\prime}^*(x)=\lim_{n\to\infty}A(F_n,\varphi^\prime)(x)=\lim_{n\to\infty}A(F_n,\varphi)(x)=\varphi^*(x)
\quad\mu\textit{-a.e.}~x\in X.
\end{gather*}
Hence the \textit{a.e.} pointwise convergence holds for $\varphi\in L^\infty(X,\mathscr{X},\mu)$ for any $\sigma$-compact amenable group.
\end{remark}

\subsection{Stability under perturbation of F{\o}lner nets}
The pointwise convergence we consider in Theorem~\ref{thm0.5} is stable under perturbation of F{\o}lner net in the following sense, which makes it is impossible to find a necessary condition of pointwise convergence for F{\o}lner net in $G$.

\begin{prop}\label{pro3.6}
Let $G\curvearrowright_TX$ be a Borel action of an amenable group $G$ on $X$, $\varphi\in L^\infty(X,\mathscr{X})$, and let $\{F_\theta; \theta\in\Theta\}$ be a F{\o}lner net in $G$ satisfying that
\begin{gather*}
\lim_{\theta\in\Theta}A(F_\theta,\varphi)(x)=\varphi^*(x)\quad \forall x\in X\ (\textrm{or \textit{a.e.} }x\in X).
\end{gather*}
Assume $\{D_\theta;\theta\in\Theta\}$ is a net of compacta of $G$ with $\lim_{\theta\in\Theta}|D_\theta|/|F_\theta|=0$, and set  $C_\theta=F_\theta\vartriangle D_\theta$. Then
\begin{gather*}
\lim_{\theta\in\Theta}A(C_\theta,\varphi)(x)=\varphi^*(x)\quad \forall x\in X\ (\textrm{or \textit{a.e.} }x\in X).
\end{gather*}
\end{prop}

\begin{proof}
Set $\varphi_\theta(x)=A(F_\theta,\varphi)(x)-A(C_\theta,\varphi)(x)$ for each $x\in X$ and $\theta\in\Theta$. Then,
\begin{gather*}
\varphi_\theta(x)=\frac{1}{|F_\theta|}\left(\int_{F_\theta}T_g\varphi(x)dg-\int_{C_\theta}T_g\varphi(x)dg\right)+\left(\frac{1}{|F_\theta|}-\frac{1}{|C_\theta|}\right)\int_{C_\theta}T_g\varphi(x)dg,
\end{gather*}
consequently, since $F_\theta\vartriangle C_\theta=D_\theta$ and $\big{|}|C_\theta|-|F_\theta|\big{|}\le|D_\theta|$,
\begin{align*}
|\varphi_\theta(x)|&\le\frac{|D_\theta|}{|F_\theta|}\left(\frac{1}{|D_\theta|}\int_{D_\theta}|T_g\varphi(x)|dg+\frac{1}{|C_\theta|}\int_{C_\theta}|T_g\varphi(x)|dg\right)\\
&\le\frac{2\|\varphi\|_\infty|D_\theta|}{|F_\theta|}.
\end{align*}
Thus, by our hypothesis, $\lim_{\theta\in\Theta}\varphi_\theta(x)=0$. This proves the proposition.
\end{proof}

Thus we may read Proposition~\ref{pro3.6} backwards to obtain that if $\{F_\theta; \theta\in\Theta\}$ and $\{S_\theta;\theta\in\Theta\}$ are two F{\o}lner nets in $G$ satisfying $\{S_\theta;\theta\in\Theta\}$ is admissible for pointwise convergence of $G\curvearrowright_TX$ and $\lim_{\theta\in\Theta}{|F_\theta\vartriangle S_\theta|}/{|S_\theta|}=0$, then $\{F_\theta; \theta\in\Theta\}$ is also admissible for pointwise convergence of $G\curvearrowright_TX$.

\subsection{The Dunford-Zygmund theorem}\label{sec3.5}
Using Theorem~\ref{thm2.2} and Theorem~\ref{thm0.6}, we can easily obtain the following $L^\infty$-pointwise ergodic theorem for product of amenable groups.

\begin{prop}
Let $G, H$ be any two $\sigma$-compact amenable groups; and let $\{F_n\}_1^\infty, \{K_n\}_1^\infty$ be two $L^\infty$-admissible F{\o}lner sequences in $G$ and $H$, respectively. Then for any Borel actions $G\curvearrowright_TX$ and $G\curvearrowright_SX$ and any $\varphi\in L^\infty(X)$,
\begin{gather*}
\lim_{n\to\infty}\frac{1}{|F_n||K_n|}\int_{F_n}\int_{K_n}T_gS_h\varphi(x)dgdh=A_T(A_S(\varphi))(x)\quad \forall x\in X.
\end{gather*}
Moreover, for any $\mu\in\mathcal{M}(G\curvearrowright_TX)\cap\mathcal{M}(H\curvearrowright_SX)$,
\begin{gather*}
A_T(A_S(\varphi))=E_\mu\left(E_\mu(\varphi|\mathscr{X}_{S,\mu})\big{|}\mathscr{X}_{G,\mu}\right)\quad a.e.
\end{gather*}
Here $A_T(\centerdot)=\lim_nA_T(F_n,\centerdot)$ and $A_S(\centerdot)=\lim_nA_S(K_n,\centerdot)$ as in Theorem~\ref{thm0.5}.
\end{prop}

This idea may be useful for $L^\infty$-pointwise ergodic theorem for connected Lie groups. See \cite{EG} and \cite[$\S7$]{Nev} for other generalizations of the Dunford-Zygmund theorem.
\section{Further applications}\label{sec4}
In this section we shall present some simple and standard applications of our mean and pointwise ergodic theorems.
Please keeping in mind that we have only convergence in the mean, which is not even convergence almost everywhere, in all of our $L^p$-mean ergodic theorems below.

\subsection{$L^p$-mean ergodic theorems}\label{sec4.1}
Another interesting application is to extend \cite[Proposition~6.6]{Fur}. Let $\sigma\colon G\rightarrow X$ be a Borel homomorphism from the amenable group $G$ to an lcH group $X$. Let $\mu$ be a fixed left-invariant $\sigma$-finite Haar measure of $X$. We then consider the dynamical system:
\begin{gather}\label{eq4.1}
T\colon G\times X\rightarrow X;\quad (g,x)\mapsto g(x)=l_{\sigma(g)}(x),
\end{gather}
where $l_y\colon x\mapsto yx$ is the left translation of $X$ by $y$ for any $y\in X$. Then, $T_g\colon \phi(x)\mapsto\phi(l_{\sigma(g)}x)$ is a unitary operator of $L^p(X,\mathscr{B}_X,\mu)$ for any $g\in G$. Next by \cite[Theorem~30C]{Loo}, for $1\le p<\infty$ and fixed $\phi\in L^p(X,\mathscr{B}_X,\mu)$, $g\mapsto T_g(\phi)$ is Borel measurable from $G$ into $(L^p(X,\mathscr{B}_X,\mu),\|\cdot\|_p)$.

Therefore Theorem~\ref{thm2.1} follows the following $L^p$-mean ergodic theorem, even though $\mu$ is not finite when $X$ is not compact.

\begin{cor}\label{cor4.1}
Let $\sigma\colon G\rightarrow X$ be a Borel homomorphism from the amenable group $G$ to an lcH group $X$. Then under $(\ref{eq4.1})$, for $1<p<\infty$, over any F{\o}lner net $(F_\theta)_{\theta\in\Theta}$ in $G$,
\begin{gather*}
L^p\textit{-}\lim_{\theta\in\Theta}\frac{1}{|F_\theta|}\int_{F_\theta}T_g\phi dg=P(\phi)\quad \forall \phi\in L^p(X,\mathscr{B}_X,\mu)
\end{gather*}
Here $P\colon L^p(X,\mathscr{B}_X,\mu)=\F_\mu^p\oplus\N_\mu^p\rightarrow\F_\mu^p$ is the projection.
\end{cor}

Proposition~6.6 of \cite{Fur} only asserts the weak convergence under the much more restricted conditions: $G\cong\mathbb{Z}^r, p=2$ and $X$ is a compact metrizable abelian group as an illuminated case of the weak mean ergodic theorem of $\mathbb{Z}^r$ \cite[Proposition~6.11]{Fur}. If $G$ is not abelian or compact, then
\begin{gather*}
\bar{\phi}:=\int_GT_g(\phi)dg\not\Rightarrow T_g(\bar{\phi})=\bar{\phi}\ \forall g\in G\quad \textrm{and} \quad\bar{\phi}\textrm{ is not well defined}.
\end{gather*}
Whence Furstenberg's proof idea of \cite[Proposition~6.6]{Fur} is not valid for our Corollary~\ref{cor4.1} above.

The above application is for a measure-preserving system. It turns out that we can do something for non-singular systems. Let $(X,\mathscr{X})$ be a Borel $G$-space. A probability measure $\mu$ on $X$ is said to be \textit{quasi-invariant} if $\mu\sim g(\mu)$, that is to say, $\mu(B)=0$ if and only if $\mu(g^{-1}[B])=0\ \forall B\in\mathscr{X}$, for all $g\in G$. Let $\rho(g,x)>0$ be the \textit{Radon-Nikodym derivative} defined by
$\rho(g,x)=\frac{dg(\mu)}{d\mu}(x)$ for $\mu$-a.e.~$x\in X$,
for any $g\in G$. Then $\rho(g,x)$ is a Borel function of the variable $(g,x)\in G\times X$ whenever $(X,\mathscr{X})$ is a standard Borel space (cf.~\cite[Example~4.2.4]{Zim}). If $\sup_{g\in G,x\in X}\rho(g,x)<\infty$, then $T_g\colon \phi\mapsto \phi\circ g$ is uniformly bounded from $L^p(X,\mathscr{X},\mu)$ to itself for $g\in G$ and hence there follows the following result:

\begin{cor}\label{cor4.2}
Let $G$ be a $\sigma$-compact amenable group and let $(X,\mathscr{X})$ be a standard Borel $G$-space with a quasi-invariant probability measure $\mu$. If Radon-Nikodym derivative satisfies $\sup_{g\in G,x\in X}\rho(g,x)<\infty$, then, for $1\le p<\infty$, any F{\o}lner sequence $\{F_n\}_1^\infty$ in $G$,
\begin{gather*}
L^p\textit{-}\lim_{n\to\infty}\frac{1}{|F_n|}\int_{F_n}T_g\phi dg=P(\phi)\quad \forall \phi\in L^p(X,\mathscr{X},\mu)
\intertext{and moreover}
\lim_{n^\prime\to\infty}\frac{1}{|F_{n^\prime}^\prime|}\int_{F_{n^\prime}^\prime}\phi(T_gx)dg=P(\phi)(x)\quad \forall x\in X\textrm{ and }\phi\in L^\infty(X,\mathscr{X})
\end{gather*}
over some F{\o}lner subsequence $\{F_{n^\prime}^\prime\}$ of $\{F_n\}$.
Here $P\colon L^p(X,\mathscr{X},\mu)=\F_\mu^p\oplus\N_\mu^p\rightarrow\F_\mu^p$
is the projection.
\end{cor}

In addition, we note that in Corollary~\ref{cor4.1} the case $p=1$ cannot be obtained from the case of $p=2$ since $\mu(X)=\infty$ whenever $X$ is not a compact group.
\subsection{Recurrence theorems for amenable groups}\label{sec4.2}
We now indicate briefly how our ergodic theorems apply to ``phenomena of physics'' as follows.
Let $G$ be an amenable group not necessarily $\sigma$-compact and let $(X,\mathscr{X},\mu)$ a probability space. We now consider that $G$ acts Borel on $X$ by $\mathscr{X}$-transformations
$$T_g\colon X\rightarrow X;\quad x\mapsto g(x),$$
which preserve $\mu$. We note that in many cases such an $G$-invariant measure always exists (cf.~\cite{Edw,Fom,Var} and Theorem~\ref{thm0.8} and Proposition~\ref{pro0.10}).
Then we may apply Theorem~\ref{thm0.4} to show that for each $\phi\in L^p(X,\mathscr{X},\mu)$,
\begin{gather}\label{eq4.2}
\frac{1}{|F_\theta|}\int_{F_\theta}T_g\phi dg\xrightarrow[]{L^p(\mu)} P(\phi)\quad \textrm{as }n\to\infty
\end{gather}
over any F{\o}lner net $\{F_\theta; \theta\in\Theta\}$ in $G$.

There holds
\begin{itemize}
\item $P\colon L^2(X,\mathscr{X},\mu)=\F_\mu^2\oplus\N_\mu^2\rightarrow\F_\mu^2$ is an orthogonal, self-adjoint and positive definite operator,
\end{itemize}
so that $\langle P(\phi),\phi\rangle\ge0$ for all $\phi\in L^2(X,\mathscr{X},\mu)$. Also $P(\mathbf{1})=\mathbf{1}$. Setting $\phi=1_\Sigma-\mu(\Sigma)$ for any $\Sigma\in\mathscr{X}$ we deduce that $\langle P(1_\Sigma),1_\Sigma\rangle\ge\mu(\Sigma)^2$. Hence for any $\Sigma\in \mathscr{X}$ there follows
$$P(1_\Sigma)(x)>0\quad \mu\textit{-a.e. }x\in\Sigma.$$
For if $B=\{x\in\Sigma\,|\,P(1_\Sigma)(x)=0\}$ we will have
$$
\mu(B)^2\le\langle P(1_B),1_B\rangle\le\langle P(1_B),1_\Sigma\rangle=\langle 1_B,P(1_\Sigma)\rangle=0.
$$

We then can obtain by Theorem~\ref{thm0.4} the following pointwise recurrence theorem, which generalizes the classical Poincar\'{e} recurrence theorem~\cite{Poi} from $\mathbb{Z}$ or $\mathbb{R}$ to amenable groups:

\begin{Thm}[Positive recurrence]\label{thm4.3}
Let $(X,\mathscr{X})$ be a Borel $G$-space and $\mu$ an invariant probability measure on it. Then for any $\Sigma\in\mathscr{X}$ and any F{\o}lner net $\mathcal{F}=\{F_\theta; \theta\in\Theta\}$ in $G$,
\begin{equation*}
\mathrm{D}_{\mathcal{F}}^*(x,\Sigma):=\limsup_{\theta\in\Theta}\frac{|\{g\in G\colon g(x)\in\Sigma\}\cap F_\theta|}{|F_\theta|}>0
\end{equation*}
for $\mu$-a.e. $x\in\Sigma$.
\end{Thm}
\begin{proof}
Let $\mu(\Sigma)>0$. Given any point $x\in\Sigma$ with $1_\Sigma^*(x)=P(1_\Sigma)(x)>0$, if $\mathrm{D}_{\mathcal{F}}^*(x,\Sigma)=0$ then
$$
\lim_{\theta\in\Theta}\frac{1}{|F_\theta|}\int_{F_\theta}1_\Sigma(T_gx)dg=0.
$$
Further by Theorem~\ref{thm0.5}, it follows that one can find a $L^\infty$-admissible F{\o}lner subnet over which $1_\Sigma^*(x)=0$, a contradiction. This proves Theorem~\ref{thm4.3}.
\end{proof}

It should be noted that since the classical $L^2$-mean ergodic theorem and Lindenstrauss's pointwise ergodic theorem is relative to a tempered $\mathscr{K}$-F{\o}lner/summing \textit{sequence}~\cite{Gre, Lin}, the above result cannot be directly concluded from those; since our F{\o}lner \textit{net} $\{F_\theta; \theta\in\Theta\}$ in $G$ as in $\S\ref{sec0.1}$ is essentially weaker than a $\mathscr{K}$-F{\o}lner/summing sequence.

\subsection{Quasi-weakly almost periodic points of metric $G$-space}
Based on the foregoing Theorem~\ref{thm4.3}, we next shall consider another generalized version of Poincar\'{e}'s recurrence theorem associated to a measurable function.

\begin{Thm}[Quasi-weakly almost periodic points]\label{thm4.4}
Let $X$ be a separable metric and $G\curvearrowright X$ a Borel action of an amenable group $G$ on $X$ preserving a Borel probability measure $\mu$. If $(Y,d)$ is a separable metric space and $\varphi$ is a measurable map of $X$ to $Y$, then $\mu$-a.e. $x\in X$ is such that to any $\epsilon>0$, $\mathbb{T}=\{g\in G\,|\,d(\varphi(x),\varphi(gx))<\epsilon\}$
has positive upper density relative to any F{\o}lner net $\{F_n; n\in\Theta\}$ in $G$.
\end{Thm}

\begin{proof}
Let $\mu_\varphi$ be the probability distribution of the random variable $\varphi$ valued in $Y$, that is to say, $\mu_\varphi=\mu\circ\varphi^{-1}$, and write its support $Y_1=\textrm{supp}(\mu_\varphi)$. Set $\mathcal{F}=\{F_n; n\in\Theta\}$.

Then for any $r>0$, $\mu(\varphi^{-1}[B(y,r)])>0$ for any $y\in Y_1$, where $B(y,r)$ is the open ball of radius $r$ centered at $y$ in $Y$. By Theorem~\ref{thm4.3}, it follows that for any $y\in Y_1$ and $\epsilon>0$ we have
$$
\mathrm{D}_\mathcal{F}^*(x,\varSigma)>0,\quad \mu\textit{-a.e.}~x\in\varSigma:=\varphi^{-1}[B(y,\epsilon)],
$$
therefore
\begin{equation}\label{eq4.3}
\mathrm{D}_\mathcal{F}^*(\{g\in G\,|\,d(\varphi(x),\varphi(gx))<\epsilon\})>0,\quad \mu\textit{-a.e.}~x\in\varSigma.
\end{equation}
By the separability of $Y$, one can find a sequence of points $y_k\in Y_1$ such that $Y_1\subseteq\bigcup_kB(y_k,\epsilon)$. Combined with (\ref{eq4.3}) this yields that there exists a set $X_\epsilon\in\mathscr{B}_X$ with $\mu(X_\epsilon)=1$ such that
\begin{equation*}
\mathrm{D}_\mathcal{F}^*(\{g\in G\,|\,d(\varphi(x),\varphi(gx))<\epsilon\})>0\quad \forall x\in X_\epsilon.
\end{equation*}
Letting $\epsilon_\ell\downarrow0$ as $\ell\to\infty$ and $Q=\bigcap_\ell X_{\epsilon_\ell}$; then $\mu(Q)=1$ and every point $x\in Q$ has the desired property.
This thus completes the proof of Theorem~\ref{thm4.4}.
\end{proof}

Let $G$ be a $\sigma$-compact amenable group. A very interesting case of the above theorem is that $\varphi$ is the identity mapping of $X$ onto itself for a separable metric Borel $G$-space $X$. Given any F{\o}lner sequence $\mathcal{F}=\{F_n; n\in\mathbb{N}\}$ in $G$, write
\begin{gather}
\mathrm{D}_\mathcal{F}^*(x,r)=\limsup_{n\to\infty}\frac{1}{|F_n|}\int_{F_n}1_{B(x,r)}(gx)dg,\quad x\in X\textrm{ and }r>0,\intertext{and}
Q_{\textit{wap}}(G\curvearrowright X)=\left\{x\in X\,|\,\mathrm{D}_\mathcal{F}^*(x,\epsilon)>0\ \forall \epsilon>0\right\}.
\end{gather}
The latter is called the set of \textbf{\textit{quasi-weakly almost periodic points}} of the Borel $G$-space $X$ relative to the F{\o}lner sequence $\{F_n; n\in\mathbb{N}\}$ in $G$ (cf., e.g.,~\cite{Zh} for the special case of $G=\mathbb{Z}$ and $F_n=\{1,2,\dotsc,n\}$ a F{\o}lner sequence in $\mathbb{Z}$). Clearly, $Q_{\textit{wap}}(G\curvearrowright X)$ is $G$-invariant if $G$ is abelian. A new ingredient of Theorem~\ref{thm4.4}, however, is that our dynamical system $T\colon G\times X\rightarrow X$ is only Borel measurable and so $T_g\colon X\rightarrow X$ is not necessarily continuous for each $g\in G$.

Since $\mathrm{D}_\mathcal{F}^*(x,\epsilon)$ is a measurable function of $x$, $Q_{\textit{wap}}(G\curvearrowright X)$ rel.~$\{F_n; n\in\mathbb{N}\}$ is $\mu$-measurable and further by Theorem~\ref{thm4.4} we easily see that $\mu(Q_{\textit{wap}}(G\curvearrowright X))=1$ for any $\mu\in\mathcal{M}(G\curvearrowright_TX)$.
Therefore, if $X$ is a compact metric $G$-space, then $Q_{\textit{wap}}(G\curvearrowright X)$ rel.~$\{F_n; n\in\mathbb{N}\}$ is of full measure $1$. So the phenomenon of quasi-weakly almost periodic motion is by no means rare. This generalizes and meanwhile strengthens the classical Birkhoff recurrence theorem.

\begin{cor}
Let $T\colon(X,\mathscr{B}_X,\mu)\rightarrow(X,\mathscr{B}_X,\mu)$ be a measure-preserving, not necessarily continuous, transformation of a separable metric space $X$. Then for $\mu$-a.e. $x\in X$, one can find $n_k\to\infty$ with $T^{n_k}x\to x$.
\end{cor}

\textit{Why do we need to introduce the concept \textit{q.w.a.p. point} for amenable group actions?} Let's recall that for any cyclic topological dynamical system $T\colon X\rightarrow X$, a point $x$ is called \textbf{\textit{recurrent}} iff $\exists n_k\to\infty$ such that $T^{n_k}x\to x$ as $k\to\infty$ (cf.~\cite{Fur}). For a group action dynamical system, $G\curvearrowright_SX$, in many literature a point $x$ is said to be recurrent if $\exists t_n\in G$ such that $S_{t_n}x\to x$ as $n\to\infty$. However, this cannot actually capture the recurrence of the orbit $G(x)$; for example, to a one-parameter flow $\mathbb{R}\curvearrowright_SX$, $S_{t_n}x\to x$ for any $t_n\to 0$ in $\mathbb{R}$.
In view of this, we need to introduce the upper density $\mathrm{D}^*(x,\epsilon)>0$ over a F{\o}lner net in $G$ instead of $n_k\to\infty$.
\section*{\textbf{Acknowledgments}}%
This work was partly supported by National Natural Science Foundation of China grant $\#$11271183 and PAPD of Jiangsu Higher Education Institutions.



\end{document}